\documentclass[12pt]{amsart}

\usepackage{hyperref}
\usepackage{amssymb, amsmath, amsthm}
\usepackage{mathrsfs}

\numberwithin{equation}{section}

\theoremstyle{plain}
\newtheorem{thm}{Theorem}[section]

\newtheorem{lemma}[thm]{Lemma}
\newtheorem{cor}[thm]{Corollary}
\theoremstyle{definition}
\newtheorem{definition}[thm]{Definition}

\newcommand{\R}{\mathbb{R}}
\newcommand{\Z}{\mathbb{Z}}

\renewcommand{\fam}{\widetilde}
\newcommand{\parent}[1]{{#1}^+}
\newcommand{\gparent}[1]{{#1}^{++}}
\newcommand{\dil}[2]{\delta_{#2}(#1)}
\newcommand{\norm}[1]{\rho({#1})}

\begin{document}

\title[Anisotropic Carleson operators]{Bounds for anisotropic Carleson operators}
\author[J. Roos]{Joris Roos}
\date{\today}
\begin{abstract}
We prove weak $(2,2)$ bounds for maximally modulated anisotropically homogeneous smooth multipliers on $\mathbb{R}^n$. These can be understood as generalizing the classical one-dimensional Carleson operator. For the proof we extend the time-frequency method by Lacey and Thiele to the anistropic setting. We also discuss a related open problem concerning Carleson operators along monomial curves.
\end{abstract}
\subjclass[2010]{42B20, 42B25, 44A12}
\keywords{Carleson operators, anisotropic dilations, time-frequency analysis, singular integrals}
\address{Mathematical Institute, University of Bonn, Endenicher Allee 60, 53115 Bonn, Germany. {\it Current Address:} Department of Mathematics, University of Wisconsin-Madison, 480 Lincoln Dr, Madison, WI-53703, USA}
\email{jroos@math.wisc.edu}

\maketitle

\section{Introduction}

Let us consider $\R^n$ equipped with the anisotropic dilations given by
\begin{equation}\label{eqn:dilation} \dil{x}{\lambda} = (\lambda^{\alpha_1} x_1, \dots, \lambda^{\alpha_n} x_n),
\end{equation}
where $\alpha=(\alpha_1,\dots,\alpha_n)$ with $\alpha_i\ge 1$ for $i=1,\dots,n$. 
We write $|\alpha|=\sum_{i=1}^n \alpha_i$ and fix the anisotropic norm
\[ \norm{x} = \max\{ |x_i|^{\frac1{\alpha_i}}\,:\,i=1,\dots,n\}. \]
 For an integer $\nu\ge 0$, we say that a function $m$ on $\R^n$ is in the class $\mathscr{M}^\nu$ if
\begin{enumerate}
\item[(a)] $m$ is bounded and contained in $C^{\nu}(\R^n\setminus\{0\})$, and
\item[(b)] $m(\dil{\xi}{\lambda})=m(\xi)$ for all $\xi\not=0$ and $\lambda>0$.
\end{enumerate}
Let us denote
\[ \|m\|_{\mathscr{M}^\nu} = \sup_{|\beta| \le \nu} \sup_{\norm{\xi}=1} |\partial^{\beta} m(\xi)|. \]
Define the Carleson operator associated with the multiplier $m$ as
\[ \mathscr{C}_m f(x) = \sup_{N\in\R^n} \left|\int_{\R^n} \widehat{f}(\xi) e^{ix\xi} m(\xi-N) d\xi \right|. \]
Replacing $\alpha$ by $c \alpha$ for some scalar $c$ does not modify the classes $\mathcal{M}^\nu$. Thus we make the assumption $\alpha_1=1$ for normalization. For technical reasons we also assume that the $\alpha_i$ are positive integers.\\

Then we can state our main result as follows.
\begin{thm}\label{thm:ac_anisocarleson} Let $\nu_0\ge 3|\alpha|+2$ be an integer. There exists $C>0$ depending only on $\alpha$ and $\nu_0$ such that for all $m\in \mathscr{M}^{\nu_0}$ we have
\begin{equation}\label{eqn:ac_main}
\|\mathscr{C}_m f\|_{2,\infty} \le C \|m\|_{\mathscr{M}^{\nu_0}} \|f\|_2.
\end{equation}
\end{thm}
The proof of this theorem is based on the time-frequency techniques of Lacey and Thiele \cite{LT00}. In the one-dimensional case $n=1, \alpha=1$ we recover the weak $(2,2)$ bound for Carleson's operator, which immediately implies Carleson's theorem on almost everywhere convergence of Fourier series \cite{Car66} (up to a standard transference argument, see \cite{KT80}). In the isotropic case $\alpha=(1,\dots,1)$ the theorem follows from a result of Sj\"olin \cite{Sjo71}. Pramanik and Terwilleger \cite{PT03}  study weak $(2,2)$ bounds for the isotropic case in $\R^n$ using the method of Lacey and Thiele. This is extended to strong $(p,p)$ for $1<p<\infty$ by Grafakos, Tao and Terwilleger in \cite{GTT04}. We speculate that Theorem \ref{thm:ac_anisocarleson} could also be extended to strong $(p,p)$ for $1<p<\infty$ using the methods from \cite{GTT04}. However we don't pursue this here to keep the exposition simple.\\

It would be interesting to know if \eqref{eqn:ac_main} holds for all $\nu_0>\frac{|\alpha|}2$, which is a natural lower bound suggested by the anisotropic H\"ormander-Mikhlin theorem (see \cite{FR66}). The same question is also open in the isotropic case. It seems plausible that \eqref{eqn:ac_main} should at least hold for all $\nu_0\ge |\alpha|+1$, because curiously, the only place in the current proof that requires more than that is a tail estimate in the single tree lemma (see \eqref{eqn:tailest}). All the main terms can be bounded using only $\nu_0\ge |\alpha|+1$.\\

The method of Lacey and Thiele involves several ingredients. The first step is a reduction to a discrete dyadic model operator that involves summation over certain regions in phase space which are called tiles. This is detailed in Section \ref{sec:ac_reduction}. In this step we encounter a complication which is caused by the absence of rotation invariance in the anisotropic case. We resolve this using an anisotropic cone decomposition (see Lemma \ref{lemma:ac_cones}). The next step is a certain procedure of combinatorial nature the purpose of which is to organize the tiles into certain collections (which are called trees) each of which is associated with a component of the operator that behaves more like a classical singular integral operator (see Lemma \ref{lemma:ac_tree}). The combinatorial part of the argument requires only little modification compared to the original procedure in \cite{LT00} (see Sections \ref{sec:ac_modelform}, \ref{sec:ac_mass}, \ref{sec:ac_energy}). 
The core component and most difficult part of the proof is the single tree estimate (see Section \ref{sec:ac_tree}). Due to an extra dependence on the linearizing function, the estimate is more technical than the corresponding estimate in \cite{LT00}. It is complicated further by the presence of anisotropic dilations.\\

In Section \ref{sec:ac_motivation} we discuss a related open problem on Carleson operators along monomial curves, which served as a main motivation for this work. We demonstrate how Theorem \ref{thm:ac_anisocarleson} can be applied to a certain family of rougher multipliers that can be seen as a toy model for the Carleson operators along monomial curves. We also discuss the particular case of the parabolic Carleson operator \eqref{eqn:parcarleson}, which exhibits some additional symmetries and is related to Lie's quadratic Carleson operator \cite{Lie09}. Some partial progress on Carleson operators along monomial curves was obtained in \cite{GPRY16}. Also see the related work \cite{PY15} on Carleson operators along paraboloids in dimensions $n\ge 3$.\\

{{\bf Acknowledgments.} The author thanks his doctoral advisor Christoph Thiele for encouragement and countless valuable comments and discussions on this project. He is also deeply indebted to Po-Lam Yung for many important and fruitful conversations. This work was carried out while the author was supported by the German Academic Scholarship Foundation.}

\section{Carleson operators along monomial curves}\label{sec:ac_motivation}
For a positive integer $d\ge 2$, let us consider the multiplier of the Hilbert transform along the curve $(t,t^d)$ in the form
\begin{equation}\label{eqn:curvemult}
m_{d}(\xi,\eta) = p.v.\int_\R e^{i\xi t-i\eta t^d} \frac{dt}t,\,\,(\xi,\eta)\in\R^2. \end{equation}
It is currently an open problem to decide whether $\mathscr{C}_{m_d}$ satisfies any $L^p$ bounds. The multiplier $m_d$ satisfies the anisotropic dilation symmetry
\[ m_d(\lambda\xi, \lambda^d \eta) = m_d(\xi,\eta) \]
for $\lambda>0$ and $(\xi,\eta)\not=0$. However, Theorem \ref{thm:ac_anisocarleson} does not apply because $m_d$ is too rough to be in the class $\mathscr{M}^{\nu}$ for any positive integer $\nu$.

Next we discuss a family of toy model operators. In the following discussion we focus on the intersection of the quadrant $\{\xi\ge 0, \eta\ge 0\}$ with the region $\eta^{\frac1d} \le 2\xi$. The other quadrants can be treated similarly (though depending on the parity of $d$ the phase might not have a critical point in each quadrant; this is an inconsequential subtlety that we will ignore). Our restriction to the region $\eta^{\frac1d} \le 2\xi$ is natural because stationary phase considerations show that $m_d$ is smooth away from the axis $\eta=0$.
For $\xi>0$ and $\eta>0$ we define
\begin{equation} 
m_{d,1}(\xi,\eta) =\left(\eta^{\frac1d}\xi^{-1}\right)^{\frac{d'}2} e^{i \left(\eta^{\frac1d}\xi^{-1}\right)^{-d'}} \psi(\eta^{\frac1d}\xi^{-1}),
\end{equation}
where $\frac1d+\frac1{d'}=1$ and $\psi$ is a smooth cutoff function supported in $[-2, 2]$ and equal to one on a slightly smaller interval. We extend $m_{d,1}$ continuously by setting it equal to zero on the remainder of $\R^2$.

From a standard computation using the stationary phase principle (see \cite[Ch. VIII.1, Prop. 3]{Ste93}) we can see that, up to a negligible constant rescaling, this term constitutes the main contribution to the oscillatory integral in \eqref{eqn:curvemult}. The remainder term from stationary phase is smoother in the variables $\xi$ and $\eta$ and is therefore simpler to handle. We will ignore it for the purpose of this discussion.

From the definition we see directly that $m_{d,1}$ (and therefore also $m_d$) is only H\"older continuous of class $C^{\frac1{2(d-1)}}$ along the axis $\eta=0$ while it is infinitely differentiable away from that axis.

Let us from here on denote 
\[\zeta=\zeta(\xi,\eta)=\eta^{\frac1d}\xi^{-1}\]
for $\xi>0, \eta>0$. Since we do not know how to handle $m_{d,1}$ we introduce a family of modified, less oscillatory multipliers which is defined on the quadrant $\{\xi>0, \eta>0\}$ by
\begin{equation} 
m_{d,\delta}(\xi,\eta) =\zeta^{\frac{d'}2} e^{i \zeta^{-d'\delta}} \psi(\zeta),
\end{equation}
where $0\le \delta\le 1$ is a parameter (and we again extend $m_{d,\delta}$ to the rest of $\R^2$ by zero). The multiplier $m_{d,\delta}$ still fails to be in $\mathscr{M}^{\nu}$ for every positive integer $\nu$ and $\delta>0$. However, we can nevertheless apply Theorem \ref{thm:ac_anisocarleson} to bound $\mathscr{C}_{m_{d,\delta}}$ for small enough $\delta$.
For this purpose we assume that $\psi$ takes the form
\[ \psi = \sum_{j\le 0} \varphi_j, \]
where $\varphi_j(x)=\varphi(2^{-j} x)$ and $\varphi$ is a smooth bump function supported in $[1/2,2]$ and satisfying $\sum_{j\in\Z} \varphi_j(x) = 1$ for all $x\not=0$.
Then we have the following consequence of Theorem \ref{thm:ac_anisocarleson}.
\begin{cor}
There exists $\delta_0>0$ such that for all $0\le \delta< \delta_0$ we have
\[ \|\mathscr{C}_{m_{d,\delta}} f\|_{2,\infty} \lesssim \|f\|_2. \]
\end{cor}

\begin{proof}
Let us write
\[m_{d,\delta,j}(\xi,\eta) =\zeta^{\frac{d'}2} e^{i \zeta^{-d'\delta}} \varphi_j(\zeta)\quad\text{ (for }\xi>0, \eta>0),\]
\[ T_j f(x,y) = \int_{\R^2} \widehat{f}(\xi,\eta) m_{d,\delta,j}(\xi,\eta) e^{ix\xi+iy\eta} d(\xi,\eta). \]
By the change of variables $\eta\mapsto 2^{jd} \eta$, we see that
\[ T_j f(x,y) = 2^{j\frac{d'}2} 2^{jd} \int_{\R^2} \widehat{f}(\xi,2^{jd}\eta) \widetilde{m}_{d,\delta,j}(\xi,\eta) e^{ix\xi+i2^{jd}y\eta} d(\xi,\eta) =2^{j\frac{d'}2} \mathrm{D}_{2^{jd}} \widetilde{T}_j \mathrm{D}_{2^{-jd}} f(x,y),\]
where $\mathrm{D}_\lambda f(x,y) = f(x,\lambda y)$ and
\[ \widetilde{m}_{d,\delta,j} (\xi,\eta) = \zeta^{\frac{d'}2} e^{i 2^{-jd'\delta} \zeta^{-d'\delta}} \varphi(\zeta)\quad\text{ (for }\xi>0, \eta>0),\]
\[ \widetilde{T}_j f(x,y) = \int_{\R^2} \widehat{f}(\xi,\eta) \widetilde{m}_{d,\delta,j}(\xi,\eta) e^{ix\xi+iy\eta} d(\xi,\eta). \]
We have 
\[ \|\widetilde{m}_{d,\delta,j}\|_{\mathscr{M}^\nu} \lesssim 2^{-jd'\delta \nu}\]
for every integer $\nu\ge 0$ (where the implied constant depends on $\nu$, $d$, $\delta$ and $\varphi$). Using Theorem \ref{thm:ac_anisocarleson} we therefore obtain
\[\|\mathscr{C}_{m_{d,\delta}}\|_{L^2\to L^{2,\infty}} \lesssim \sum_{j\le 0} 2^{j\frac{d'}2} \|\mathscr{C}_{\widetilde{m}_{d,\delta,j}} \|_{L^2\to L^{2,\infty}} 
\lesssim \sum_{j\le 0} 2^{jd'(\frac12-\delta \nu_0)}.\]
Thus, setting $\delta_0 = \frac12 \nu_0^{-1}$ yields the claim.
\end{proof}
An improvement for the bound on $\nu_0$ in Theorem \ref{thm:ac_anisocarleson} will give an improvement for $\delta_0$. However, even if we could show Theorem \ref{thm:ac_anisocarleson} for all $\nu_0>|\alpha|/2$, we would still have $\delta_0<\frac1{|\alpha|}<1$. Therefore additional insight is likely to be required to bound the operator $\mathscr{C}_{m_{d,1}}$ (and thus $\mathscr{C}_{m_d}$).

For the case of the parabola, $d=2$, there are some additional obstructions in bounding $\mathscr{C}_{m_2}$. Let us write the parabolic Carleson operator as
\begin{equation}\label{eqn:parcarleson}
\mathscr{C}^{\mathrm{par}} f(x,y) = \sup_{N\in\R^2} \left| p.v.\int_\R f(x-t, y-t^2) e^{iN_1 t+ i N_2 t^2} \frac{dt}t\right|.
\end{equation}
Apart from the linear modulation symmetries given by
\[ \mathscr{C}^{\mathrm{par}} f = \mathscr{C}^{\mathrm{par}} \mathrm{M}_{N} f \]
for $N\in\R^2$, there are additional modulation symmetries. For a polynomial in two variables, $P=P(x,y)$, we write the corresponding polynomial modulation as
\[ \mathrm{M}_{P} f(x,y) = e^{iP(x,y)} f(x,y).\]
Then we have that
\begin{equation}\label{eqn:ac_quadraticmod}
\mathscr{C}^{\mathrm{par}} \mathrm{M}_{N x^2} f = \mathscr{C}^{\mathrm{par}} f,
\end{equation}
\[ \mathscr{C}^{\mathrm{par}} \mathrm{M}_{N x(y+x^2)} f = \mathscr{C}^{\mathrm{par}} f, \]
\[ \mathscr{C}^{\mathrm{par}} \mathrm{M}_{N (y+x^2)^2} f = \mathscr{C}^{\mathrm{par}} f \]
hold for all $N\in\R$. The quadratic modulation symmetry \eqref{eqn:ac_quadraticmod} suggests a connection to Lie's quadratic Carleson operator \cite{Lie09}. Indeed, even a certain partial $L^2$ bound for $\mathscr{C}^{\mathrm{par}}$ would immediately imply an $L^2$ bound for the quadratic Carleson operator (see \cite{GPRY16}).
These are all the polynomial modulation symmetries of the operator $\mathscr{C}^{\mathrm{par}}$ (up to linear combination). To see that, introduce the change of variables $\tau(x,y)=(x,y+x^2)$ and observe that
\[\mathscr{C}^{\mathrm{par}}f=\mathscr{C}^{\mathrm{sh}}(f\circ\tau^{-1})\circ \tau,\]
where
\[\mathscr{C}^{\mathrm{sh}}f(x,y)=\sup_{N\in\R^2} \left|p.v.\int_\R f(x-t,y-2xt) e^{iN_1 t + iN_2 t^2} \frac{dt}t \right|.\]
It is easy to check that we have
\[ \mathscr{C}^{\mathrm{sh}} \mathrm{M}_P f = \mathscr{C}^{\mathrm{sh}}f \]
for a polynomial $P$ if and only if $P$ is of degree at most two. This shows that the list of polynomial modulation symmetries that we gave for $\mathscr{C}^{\mathrm{par}}$ is complete.
Also since $\tau$ is measure preserving, $L^p$ bounds for $\mathscr{C}^{\mathrm{par}}$ and $\mathscr{C}^{\mathrm{sh}}$ are equivalent.

\section{Reduction to a model operator}\label{sec:ac_reduction}

Before we begin we need to introduce some more notation and definitions. 
Denote
\[ \mathrm{dist}_\alpha (A,B) = \inf_{x\in A, y\in B} \norm{x-y}\]
and $\mathrm{dist}_\alpha(A,x)=\mathrm{dist}_\alpha(A,\{x\})$. For $a,b\in\R^n$ we write 
\[ [a,b] = \prod_{i=1}^n [a_i, b_i] \]
and similarly $(a,b), [a,b)$. We will refer to all such sets as \emph{rectangles}. For a rectangle $I\subset\R^n$ we define $c(I)$ to be its center. By an \emph{anisotropic cube} we mean a rectangle $[a,b]$ such that $b_i-a_i=\lambda^{\alpha_i}$ holds for all $i=1,\dots,n$ and some $\lambda>0$.
We define the collection of \emph{anisotropic dyadic cubes} by
\[ \mathcal{D}^{\alpha} = \left\{ [\dil{\ell}{2^k}, \dil{\ell+1}{2^k} )\,:\,\ell\in\Z^n,\,k\in\Z\right\}.\]
Every two anisotropic dyadic cubes have the property that they are either disjoint or contained in one another. Moreover, for every $I\in\mathcal{D}^\alpha$ there exists a unique dyadic cube $\parent{I}\in\mathcal{D}^\alpha$ such that $|\parent{I}|=2^{|\alpha|}|I|$ and $I\subset \parent{I}$. We  call $\parent{I}$ the \emph{parent} of $I$ and say that $I$ is a \emph{child} of $\parent{I}$. 
\begin{definition}A \emph{tile} $P$ is a rectangle in $\R^{n}\times \R^n$ of the form
\[ P = I_P \times \omega_P, \]
where $I_P,\omega_P\in \mathcal{D}^{\alpha}$ and $|I_P|\cdot |\omega_P|=1$. 
\end{definition}
The set of tiles is denoted by $\overline{\mathcal{P}}$. Given a tile $P$ we denote its \emph{scale} by $k_P = |I_P|^{1/|\alpha|}$. For $r\in\{0,1\}^n$ and a tile $P$ with $\omega_P = [\dil{\ell}{2^{-k_P}}, \dil{\ell+1}{2^{-k_P}}]$ we define the \emph{semi-tile} $P(r)$ by
\[ P(r) = I_P \times \omega_{P(r)}, \text{ where } \omega_{P(r)} = \Big[ \dil{\ell+\frac12 r}{2^{-k_P}}, \dil{\ell+\frac12 (r+1)}{2^{-k_P}} \Big].\]
The model operator is built up using a large family of wave packets adapted to tiles. It is convenient to generate this family by letting the symmetry group of our operator act on a single bump function. For this purpose, let $\phi$ be a Schwartz function on $\R^n$ such that $0\le \widehat{\phi}\le 1$ with $\widehat{\phi}$ being supported in $[-\frac{b_0}2, \frac{b_0}2]^n$ and equal to $1$ on $[-\frac{b_1}2, \frac{b_1}2]^n$, where $0<b_1<b_0\ll 1$ are some fixed, small numbers whose ratio is not too large (it becomes clear what precisely is required in Section \ref{sec:ac_tree}). For example, we may set $b_0=\frac1{10}$, $b_1=\frac{9}{100}$.
We denote translation, modulation and dilation of a function $f$ by
\[ \mathrm{T}_y f(x) = f(x-y),\quad\quad (y\in\R^n)\]
\[ \mathrm{M}_\xi f(x) = e^{ix\xi} f(x),\quad\quad (\xi\in\R^n)\]
\[ \mathrm{D}_\lambda^p f(x) = \lambda^{-\frac{|\alpha|}p} f\left(\dil{x}{\lambda^{-1}}\right),\quad\quad(\lambda,p>0),\]
where $|\alpha|=\sum_{i=1}^n \alpha_i$.

Given a tile $P$ and $N\in\R^n$ we define the wave packets $\phi_P, \psi^N_P$ on $\R^n$ by
\begin{align}
\phi_P(x) &= \mathrm{M}_{c(\omega_{P(0)})} \mathrm{T}_{c(I_P)} \mathrm{D}^2_{2^{k(P)}} \phi(x)\\
\widehat{\psi^N_P}(\xi) &= \mathrm{T}_{N}m(\xi)\cdot\widehat{\phi_P}(\xi)
\end{align}
We think of $\phi_P$ as being essentially time-frequency supported in the semi-tile $P(0)$. More precisely, we have that $\widehat{\phi_P}$ is compactly supported in (a small cube centrally contained in) $\omega_{P(0)}$ and $|\phi_P|$ decays rapidly outside of $I_P$.

For $N\in\R^n$ and $r\not=0$ we introduce the dyadic model sum operator 
\begin{align}\label{eqn:ac_modelop}
A^{r,m}_N f(x)=\sum_{P\in\mathcal{\overline{P}}} \langle f,\phi_P\rangle \psi^N_P(x) \mathbf{1}_{\omega_{P(r)}}(N).
\end{align}
This reduces to the model sum of Lacey and Thiele \cite{LT00} in the case $n=\alpha=1$ and to that of Pramanik and Terwilleger \cite{PT03} in the isotropic case $\alpha=(1,\dots,1)$.

\begin{thm}\label{thm:ac_anisocmodel}
For every large enough integer $\nu_0$ there exists $C>0$ depending only on $\nu_0$, $\alpha$ and the choice of $\phi$ such that for all multipliers $m\in \mathscr{M}^{\nu_0}$ we have
\begin{equation}\label{eqn:anisocmodelbound}
\| \sup_{N\in\R^n} |A^{r,m}_{N} f| \|_{2,\infty} \le C \|m\|_{\mathscr{M}^{\nu_0}} \|f\|_2.
\end{equation}
\end{thm}

The proof of the theorem is contained in Sections \ref{sec:ac_modelform}, \ref{sec:ac_mass}, \ref{sec:ac_energy}, \ref{sec:ac_tree}. 
We conclude this section by showing that Theorem \ref{thm:ac_anisocmodel} implies Theorem \ref{thm:ac_anisocarleson}. For this purpose we employ the averaging procedure of Lacey and Thiele \cite{LT00} combined with an anisotropic cone decomposition of the multiplier $m$. By an \emph{anisotropic cone} we mean a subset $\Theta\subsetneq\R^n$ of the form
\[ \Theta=\{\dil{\xi}{t}\,:\,t>0,\,\xi\in Q\} \]
for some cube $Q\subset\R^n$.
Let us denote $\mathcal{B}_s=\{x\,:\,\norm{x}\le s\}$. Let
\begin{equation}\label{eqn:ac_average}
\mathbf{A}^{r,m} f(x) = \lim_{R\to\infty} \frac1{R^{2|\alpha|}} \int_{\mathcal{B}_R} \int_{\mathcal{B}_R} \int_0^1 \mathrm{M}_{-\eta}\mathrm{T}_{-y}\mathrm{D}^2_{2^{-s}} A^{r,m}_{2^{-s}\eta} \mathrm{D}^2_{2^s} \mathrm{T}_y \mathrm{M}_\eta f(x) ds dy d\eta.
\end{equation}
\begin{lemma}\label{lemma:ac_cones} For every $r\in\{0,1\}^n$ and every test function $f$, the function $\mathbf{A}^{r,m} f(x)$ is well-defined and also a test function. We have
\[ \widehat{\mathbf{A}^{r,m} f}(\xi) = \theta_r(\xi) m(\xi) \widehat{f}(\xi) \]
for some smooth function $\theta_r$ that is independent of $m$. Moreover, there exists a constant $\varepsilon_0>0$ and an anisotropic cone $\Theta_r$ such that
\[ \theta_r(\xi) > \varepsilon_0\quad\text{for all }\xi\in\Theta_r.\]
and
\begin{equation}\label{eqn:ac_conescovering}
(-\infty,\varepsilon_0]^n \subset \bigcup_{r\in \{0,1\}^n\setminus\{0\}} \Theta_r.
\end{equation}
\end{lemma}

\begin{proof}
By expanding definitions we see that
\[ (\mathrm{M}_{-\eta}\mathrm{T}_{-y}\mathrm{D}^2_{2^{-s}} A^{r,m}_{2^{-s}\eta} \mathrm{D}^2_{2^s} \mathrm{T}_y \mathrm{M}_\eta f)^\wedge (\xi) \]
is equal to a universal constant times
\begin{align*}
m(\xi) \sum_{P\in\overline{\mathcal{P}}} & \langle \widehat{f}, \mathrm{T}_{-\eta+\dil{c(\omega_{P(0)})}{2^s}} \mathrm{M}_{y-\dil{c(I_P)}{2^{-s}}} \mathrm{D}^2_{2^{s-k_P}} \widehat{\phi} \rangle\\
& \times \mathrm{T}_{-\eta+\dil{c(\omega_{P(0)})}{2^s}} \mathrm{M}_{y-\dil{c(I_P)}{2^{-s}}} \mathrm{D}^2_{2^{s-k_P}} \widehat{\phi}(\xi) \mathbf{1}_{\omega_{P(r)}}(\dil{\eta}{2^{-s}}),
\end{align*}
where we have used that $m(\dil{\xi}{2^{-s}})=m(\xi)$. The previous display equals
\begin{align*}
m(\xi)\sum_{k\in\Z} \sum_{\ell\in\Z^n} \sum_{u\in\Z^n} 2^{-|\alpha|(s-k)} \int_{\R^n} & \widehat{f}(\zeta) e^{i(y-\dil{u+\frac12}{2^{-s+k}})(\xi-\zeta)} \overline{\widehat{\phi}} \Big(\dil{\zeta+\eta}{2^{-s+k}}-\Big(\ell+\frac14\Big)\Big) d\zeta\\
& \times \widehat{\phi} \Big(\dil{\xi+\eta}{2^{-s+k}}-\Big(\ell+\frac14\Big)\Big) \mathbf{1}_{\omega_{P(r)}}(\dil{\eta}{2^{-s}}).
\end{align*}
Applying the Poisson summation formula to the summation in $u$ and using the Fourier support information of the function $\phi$ we see that the previous display equals (up to a universal constant)
\begin{align*}
m(\xi)\widehat{f}(\xi) \sum_{k\in\Z} \sum_{\ell\in\Z^n} |\widehat{\phi}|^2 \Big(\dil{\xi+\eta}{2^{-s+k}}-\Big(\ell+\frac14\Big)\Big) \mathbf{1}_{\omega_{P(r)}}(\dil{\eta}{2^{-s}}).
\end{align*}
Observe that the expression no longer depends on the variable $y$.
It remains to compute the function $\theta_r(\xi)=c\cdot \lim_{R\to\infty} I_R(\xi)$, where $c$ is a universal constant and
\[ I_R(\xi) = \frac1{R^{|\alpha|}} \int_{\mathcal{B}_R} \int_0^1 \sum_{k\in\Z} \sum_{\ell\in\Z^n} |\widehat{\phi}|^2 \Big(\dil{\xi+\eta}{2^{-s+k}}-\Big(\ell+\frac14\Big)\Big) \mathbf{1}_{\omega_{P(r)}}(\dil{\eta}{2^{-s}}) ds d\eta.\]
Note the formula
\[ \int_0^1 \sum_{k\in\Z} F(2^{k-s}) ds = \frac1{\log 2} \int_0^\infty F(t) \frac{dt}t, \]
which follows from a change of variables $2^{k-s}\to t$. Using this we have
\[ I_R(\xi) = \frac{c}{R^{|\alpha|}} \int_{\mathcal{B}_R} \int_0^\infty \sum_{\ell\in\Z^n} |\widehat{\phi}|^2 \Big(\dil{\xi+\eta}{t}-\Big(\ell+\frac14\Big)\Big) \mathbf{1}_{Q_r}(\dil{\eta}{t}-\ell) \frac{dt}t d\eta,\]
where $Q_r=\Big[\frac12 r,\frac12 (r+1)\Big]=\prod_{i=1}^n\Big[\frac12 r_i, \frac12 (r_i+1)\Big]$ and $c=(\log 2)^{-1}$ ($c$ may change from line to line in this proof). To simplify our expression further we perform the change of variables
\[ \dil{\xi+\eta}{t}-\ell \to \zeta \]
in the integration in $\eta$. This yields
\begin{equation}\label{eqn:ac_reductioncalc}
I_R(\xi) = c\int_{\R^n} \int_0^\infty \chi(\zeta) \mathbf{1}_{Q_r}(\zeta-\dil{\xi}{t}) \left(\sum_{\ell\in\Z^n} \frac{\mathbf{1}_{\norm{\zeta+\ell-\dil{\xi}{t}}\le tR}}{(tR)^{|\alpha|}} \right)\frac{dt}{t} d\zeta
\end{equation}
where we have set
\[ \chi(\zeta) = |\widehat{\phi}|^2 \Big( \zeta - \frac14\Big). \]
Observe that the integrand in \eqref{eqn:ac_reductioncalc} is supported in a compact subset of $\R^n\times (0,\infty)$ (which depends on $\xi$).
By counting the $\ell$ for which the summand is non-zero we see that for every fixed $\zeta, \xi\in\R^n$ and $t>0$ the sum
\[ \sum_{\ell\in\Z^n} \frac{\mathbf{1}_{\norm{\zeta+\ell-\dil{\xi}{t}}\le tR}}{(tR)^{|\alpha|}} \]
converges to a universal constant as $R\to\infty$. Thus, from Lebesgue's dominated convergence theorem we conclude that
\begin{equation}\label{eqn:ac_theta}
\theta_r(\xi) = c \int_{\R^n} \int_0^\infty \chi(\zeta) \mathbf{1}_{Q_r}(\zeta-\dil{\xi}{t})\frac{dt}{t} d\zeta.
\end{equation}
Evidently we have $\theta_r(\dil{\xi}{t})=\theta_r(\xi)$ for every $t>0$ and $\xi\in\R^n$. From our choice of $\phi$ we get that $\chi$ is supported on $Q^{(0)}$ and equal to one on $Q^{(1)}$, where
\[ Q^{(j)} = \Big[\frac14 - \frac{b_j}2,\frac14 + \frac{b_j}2\Big] \] 
for $j=0,1$. Let us set
\[ \Theta_r = \{ \dil{\xi}{t}\,:\, \xi\in Q^{(1)}-Q_r \}. \]
Then we can read off \eqref{eqn:ac_theta} that $\theta_r$ is greater than some positive constant on $\Theta^{(1)}_r$. Note that
\[ Q^{(j)}-Q_r = \Big[-\frac12r - \Big(\frac14 + \frac{b_j}2\Big), -\frac12r + \Big(\frac14 + \frac{b_j}2\Big)\Big].\]
Looking at the anisotropic cone generated by each of the regions $Q^{(1)}-Q_r$ we see that \eqref{eqn:ac_conescovering} is satisfied for sufficiently small $\varepsilon_0$.
\end{proof}

In the isotropic case $\alpha=(1,\dots,1)$ we can assume without loss of generality that the multiplier $m$ is supported in some arbitrarily chosen cone. Due to the lack of rotation invariance this assumption becomes invalid in the anisotropic setting. 

\begin{proof}[Proof of Theorem \ref{thm:ac_anisocarleson}]
Let $m\in\mathscr{M}^{\nu_0}$. Without loss of generality we may assume that $m$ is supported in the ``quadrant'' $(-\infty,0]^n$. By \eqref{eqn:ac_conescovering} we can choose smooth functions $(\varrho_r)_r$ such that $\varrho_r$ is supported in $\Theta_r$ and
\[ \sum_{r\in\{0,1\}^n\setminus\{0\}} \varrho_r(\xi) = 1 \]
for $\xi\in (-\infty,0]^n$. By the triangle inequality and Lemma \ref{lemma:ac_cones}, we have
\[ \|\mathscr{C}_m f\|_{2,\infty} \le \sum_{r\in\{0,1\}^n\setminus\{0\}} \| \sup_{N\in\R^n} |\mathbf{A}^{r,\theta_r^{-1}\varrho_r m} \mathrm{M}_{N} f| \|_{2,\infty}. \]
Here $\theta_r^{-1}$ refers to the function $\xi\mapsto (\theta_r(\xi))^{-1}$, which is bounded on $\Theta_r$. By \eqref{eqn:ac_average} and Minkowski's integral inequality, the previous is no greater than
\[ \sum_{r\in\{0,1\}^n\setminus\{0\}} \limsup_{R\to\infty} \frac1{R^{2|\alpha|}} \int_{\mathcal{B}_R}\int_{\mathcal{B}_R} \int_0^1 \| \sup_{N\in\R^n} |A_N^{r,\theta_r^{-1}\varrho_r m} \mathrm{D}^2_{2^s} \mathrm{T}_y \mathrm{M}_\eta f | \|_{2,\infty} ds dy d\eta, \]
which by Theorem \ref{thm:ac_anisocmodel} is bounded by
\[ C \sum_{r\in\{0,1\}^n\setminus\{0\}} \| \theta_r^{-1}\varrho_r m\|_{\mathscr{M}^{\nu_0}} \|f\|_2\lesssim \|m\|_{\mathscr{M}^{\nu_0}} \|f\|_2. \]
\end{proof}

\section{Boundedness of the model operator}\label{sec:ac_modelform}
In this section we describe the proof of Theorem \ref{thm:ac_anisocmodel}. We follow \cite{LT00}. First, we perform some preliminary reductions. Given a measurable function $N:\R^n\to \R^n$ we define
\[ Tf(x) = A^r_{N(x)} f(x). \]
Note that the estimate \eqref{eqn:anisocmodelbound} is equivalent to showing
\[ \|T f\|_{2,\infty} \le C \|m\|_{\mathscr{M}^{\nu_0}} \|f\|_2 \]
with $C$ not depending on the choice of the measurable function $N$. By duality, it is equivalent to show
\[ |\langle Tf, \mathbf{1}_E \rangle| \lesssim \|m\|_{\mathscr{M}^{\nu_0}} |E|^{\frac12} \|f\|_2, \]
where $E$ is an arbitrary measurable set. 
By scaling, we may assume without loss of generality that $\|f\|_2=1$ and $|E|\le 1$.
Thus, by the triangle inequality, it suffices to show that
\begin{equation}\label{eqn:ac_keymodelest}
\sum_{P\in\mathcal{P}} |\langle f,\phi_P\rangle \langle \mathbf{1}_{E\cap N^{-1}(\omega_{P(r)})}, \psi^{N(\cdot)}_P\rangle| \lesssim \|m\|_{\mathscr{M}^{\nu_0}}, 
\end{equation}
for all finite sets of tiles $\mathcal{P}\subset\overline{\mathcal{P}}$, with the implied constant being independent of $f,E,N,\mathcal{P}$.
Throughout this and the following sections we fix $r\in\{0,1\}^n\setminus\{0\}$.
Before we continue we need to introduce certain collections of tiles called trees. There is a partial order on tiles defined by
\[ P\le P'\quad\quad\text{if}\quad\quad I_P\subset
I_{P'}\quad\text{and}\quad c(\omega_{P'})\in \omega_P. \]
Observe that two tiles are comparable with respect to $\le$ if and only if they have a non-empty intersection.
\begin{definition}\label{def:ac_tree}
A finite collection $\mathbf{T}\subset\overline{\mathcal{P}}$ of tiles is called a \emph{tree} if there exists $P\in\mathbf{T}$ such that $P'\le P$ for every $P'\in\mathbf{T}$. In that case, $P$ is uniquely determined and referred to as the \emph{top} of the tree $\mathbf{T}$. We denote the top of a tree $\mathbf{T}$ by $P_\mathbf{T}=I_\mathbf{T}\times \omega_\mathbf{T}$ and write $k_{\mathbf{T}}=|I_\mathbf{T}|^{1/|\alpha|}$.\\
A tree $\mathbf{T}$ is called a \emph{$1$--tree} if $c(\omega_\mathbf{T})\not\in \omega_{P(r)}$ for all $P\in\mathbf{T}$ and it is called a \emph{$2$--tree} if $c(\omega_\mathbf{T})\in \omega_{P(r)}$ for all $P\in\mathbf{T}$. These names are due to historical reasons (see \cite{LT00}).
\end{definition}
The notion of a tree was first introduced in this context by C. Fefferman \cite{Fef73}.
For a tile $P\in\overline{\mathcal{P}}$ we write 
\[E_{P}=E\cap N^{-1}(\omega_{P})\text{ and } E_{P(r)}=E\cap N^{-1}(\omega_{P(r)}).\]
The \emph{mass} of a single tile $P$ is defined as
\begin{align}
\label{eqn:ac_massdefsingletile}
\mathcal{M}(P)=\sup_{P^\prime\ge P} \int_{E_{P^\prime}}
w^{\nu_1}_{P^\prime}(x)dx,
\end{align}
where $\nu_1$ is a fixed large positive number depending only on $|\alpha|$ that is to be determined later and \[ w^{\nu}_P(x) = \mathrm{T}_{c(I_P)} \mathrm{D}^1_{2^{k(P)}} w^{\nu}(x),\]
where the weight $w^{\nu}$ takes the form
\[w^{\nu}(x)=(1+\norm{x})^{-\nu}. \]
For convenience we also write $w_P=w^{\nu_1}_P$. For a collection of tiles $\mathcal{P}\subset\mathcal{\overline{P}}$ we define
their mass as
\begin{align}
\label{eqn:ac_massdef}
\mathcal{M}(\mathcal{P})=\sup_{P\in\mathcal{P}}
\mathcal{M}(P).
\end{align}
The \emph{energy} of a collection of tiles $\mathcal{P}$ is defined as
\begin{align}
\label{eqn:ac_energydef}
\mathcal{E}(\mathcal{P})=\sup_{\mathbf{T}\subset\mathcal{P}\,\,2\mathrm{-tree}}
\left(\frac{1}{|I_T|} \sum_{P\in \mathbf{T}} |\langle f,\phi_P\rangle|^2\right)^{1/2}.
\end{align}
These quantities and the following lemmas originate in \cite{LT00}. 

\begin{lemma}[Mass lemma]\label{lemma:ac_mass} If $\nu_1>|\alpha|+1$, then there exists $C>0$ depending only on $\alpha$ such that for every finite set of tiles $\mathcal{P}\subset\overline{\mathcal{P}}$ there is a decomposition $\mathcal{P}=\mathcal{P}_{\mathrm{light}}\cup\mathcal{P}_{\mathrm{heavy}}$ such that
\begin{align}
\label{eqn:ac_masslemma1}
\mathcal{M}(\mathcal{P}_{\mathrm{light}})\le 2^{-2}\,\mathcal{M}(\mathcal{P})
\end{align}
and $\mathcal{P}_{\mathrm{heavy}}$ is a union of a set $\mathcal{T}$ of trees such that
\begin{align}
\label{eqn:ac_masslemma2}
\sum_{\mathbf{T}\in\mathcal{T}} |I_{\mathbf{T}}| \le \frac{C}{\mathcal{M}(\mathcal{P})}.
\end{align}
\end{lemma}

\begin{lemma}[Energy lemma]\label{lemma:ac_energy}
There exists $C>0$ depending only on $\alpha$ such that for every finite set of tiles $\mathcal{P}\subset\overline{\mathcal{P}}$ there is a decomposition $\mathcal{P}=\mathcal{P}_{\mathrm{low}}\cup\mathcal{P}_{\mathrm{high}}$ such that
\begin{align}
\label{eqn:ac_energylemma1}
\mathcal{E}(\mathcal{P}_{\mathrm{low}})\le 2^{-1}\,\mathcal{E}(\mathcal{P}) \end{align}
and $\mathcal{P}_{\mathrm{high}}$ is a union of a set $\mathcal{T}$ of trees such that
\begin{align}
\label{eqn:ac_energylemma2}
\sum_{\mathbf{T}\in\mathcal{T}} |I_\mathbf{T}| \le \frac{C}{\mathcal{E}(\mathcal{P})^2}.
\end{align}
\end{lemma}

\begin{lemma}[Tree estimate]\label{lemma:ac_tree}
There exists $C>0$ depending only on $\alpha$ such that if $m\in\mathscr{M}^{\nu_0}$, then the following inequality holds for every tree $\mathbf{T}$:
\begin{align}\label{eqn:ac_treeest}
\sum_{P\in \mathbf{T}} |\langle f,\phi_P\rangle\langle \psi_P^{N(\cdot)}, \mathbf{1}_{E_{P(r)}}
\rangle| \le C \|m\|_{\mathscr{M}^{\nu_0}}
|I_\mathbf{T}|\mathcal{E}(\mathbf{T})\mathcal{M}(\mathbf{T})
\end{align}
\end{lemma}

The proofs of these lemmas are contained in Sections \ref{sec:ac_mass}, \ref{sec:ac_energy}, \ref{sec:ac_tree}, respectively. By iterated application of these lemmas we obtain a proof of \eqref{eqn:ac_keymodelest}. This argument is literally the same as in \cite{LT00}, but we include it here for convenience of the reader.
Let $\mathcal{P}$ be a finite collection of tiles. We will decompose $\mathcal{P}$ into disjoint sets $(\mathcal{P}_\ell)_{\ell\in\mathcal{N}}$ (where $\mathcal{N}$ is some finite set of integers) such that each $\mathcal{P}_\ell$ satisfies
\begin{equation}\label{eqn:ac_inductionprop1}
\mathcal{M}(\mathcal{P}_\ell) \le 2^{2\ell}\quad\text{and}\quad
\mathcal{E}(\mathcal{P}_\ell)\le 2^\ell
\end{equation}
and is equal to the union of a set of trees $\mathcal{T}_\ell$ such that
\begin{equation}\label{eqn:ac_inductionprop2}
\sum_{\mathbf{T}\in\mathcal{T}_\ell} |I_\mathbf{T}| \le C 2^{-2\ell}.
\end{equation}
This is achieved by the following procedure:
\begin{enumerate}
\item[(1)] Initialize $\mathcal{P}^{\mathrm{stock}}:=\mathcal{P}$ and choose an initial $\ell$ that is large enough such that 
\begin{equation}\label{eqn:ac_inductionprop3}
\mathcal{M}(\mathcal{P}^{\mathrm{stock}}) \le 2^{2\ell}\quad\text{and}\quad\mathcal{E}(\mathcal{P}^{\mathrm{stock}})\le 2^\ell.
\end{equation}
\item[(2)] If $\mathcal{M}(\mathcal{P}^{\mathrm{stock}})>2^{2(\ell-1)}$, then apply Lemma \ref{lemma:ac_mass} to decompose $\mathcal{P}^{\mathrm{stock}}$ into $\mathcal{P}_{\mathrm{light}}$ and $\mathcal{P}_{\mathrm{heavy}}$. We add\footnote{We can think of all the $\mathcal{P}_\ell$ as being initialized by the empty set.} $\mathcal{P}_{\mathrm{heavy}}$ to $\mathcal{P}_\ell$ and update $\mathcal{P}^{\mathrm{stock}}:=\mathcal{P}_{\mathrm{light}}$ (thus, we now have $\mathcal{M}(\mathcal{P}^{\mathrm{stock}})\le 2^{2(\ell-1)}$).
\item[(3)] If $\mathcal{E}(\mathcal{P}^{\mathrm{stock}})>2^{\ell-1}$, then apply Lemma \ref{lemma:ac_energy} to decompose $\mathcal{P}^{\mathrm{stock}}$ into $\mathcal{P}_{\mathrm{low}}$ and $\mathcal{P}_{\mathrm{high}}$. We add $\mathcal{P}_{\mathrm{high}}$ to $\mathcal{P}_\ell$ and update $\mathcal{P}^{\mathrm{stock}}:=\mathcal{P}_{\mathrm{low}}$  (thus, we now have $\mathcal{E}(\mathcal{P}^{\mathrm{stock}})\le 2^{\ell-1}$).
\item[(4)] If $\mathcal{P}^{\mathrm{stock}}$ is not empty, then replace $\ell$ by $\ell-1$ and go to Step (2).
\end{enumerate}
Then we can finish the proof of \eqref{eqn:ac_keymodelest} by using \eqref{eqn:ac_inductionprop1}, \eqref{eqn:ac_inductionprop2}, \eqref{eqn:ac_treeest} and keeping in mind that we always have $\mathcal{M}(\mathcal{P})\le \|w^{\nu_1}\|_1$:
\begin{align*}
\sum_{P\in\mathcal{P}} & |\langle f,\phi_P\rangle \langle \mathbf{1}_{E\cap N^{-1}(\omega_{P(r)})}, \psi^{N(\cdot)}_P\rangle| = \sum_{\ell\in\mathcal{N}} \sum_{\mathbf{T}\in\mathcal{T}_\ell} \sum_{P\in\mathbf{T}} |\langle f,\phi_P\rangle \langle \mathbf{1}_{E\cap N^{-1}(\omega_{P(r)})}, \psi^{N(\cdot)}_P\rangle|\\
&\lesssim \|m\|_{\mathscr{M}^{\nu_0}} \sum_{\ell\in\mathcal{N}} 2^\ell \mathrm{min}(1,2^{2\ell}) \sum_{\mathbf{T}\in\mathcal{T}_\ell} |I_\mathbf{T}| \lesssim \|m\|_{\mathscr{M}^{\nu_0}}\sum_{\ell\in\Z} 2^{-\ell} \mathrm{min}(1,2^{2\ell})\lesssim \|m\|_{\mathscr{M}^{\nu_0}}.
\end{align*}

To conclude this section we collect several standard auxiliary estimates for $m,K,\phi_P,\psi_P^N$ which are used during the remainder of the proof. 
First, from the definition of $\mathscr{M}^{\nu}$ we have the symbol estimate
\begin{equation}\label{eqn:mderivest}
|\partial_i^\nu m(\xi)| \le \|m\|_{\mathscr{M}^{\nu}} \norm{\xi}^{-\nu\alpha_i}
\end{equation}
for every integer $\nu\le\nu_0$ and $i=1,\dots,n$. If we let $K$ denote the corresponding kernel (that is, $\widehat{K}=m$), we have
\begin{equation}\label{eqn:aniso_kernelest}
|K(x)| \lesssim \|m\|_{\mathscr{M}^{\lfloor\frac{|\alpha|}2\rfloor+1}} \norm{x}^{-|\alpha|}
\end{equation}
for $x\not=0$. This is a consequence of the anisotropic H\"ormander-Mikhlin theorem (see \cite{FR66}). For every integer $\nu\ge 0$ and $N\not\in\omega_{P(0)}$ we have
\begin{align}
\label{eqn:psiestimate}
|\psi_P^N(x)| &\lesssim \|m\|_{\mathscr{M}^{\nu}} |I_P|^{1/2} w_P^{\nu}(x),\end{align}
where the implicit constant depends only on $\nu, \alpha$ and the choice of $\phi$. We defer the proof of this estimate to Section \ref{sec:ac_wavepacketest}.

The next estimates concern the interaction of two wave packets associated with distinct tiles. Let $P,P'~\in~\overline{\mathcal{P}}$ be tiles. The idea is that if $P,P'$ are disjoint (or equivalently, incomparable with respect to $\le$) then their associated wave packets are almost orthogonal, i.e. $\langle \phi_P, \phi_{P'}\rangle$ is negligibly small. Indeed, if $\omega_{P}$ and $\omega_{P'}$ are disjoint, then we even have $\langle \phi_P, \phi_{P'}\rangle=0$. However, as an artifact of the Heisenberg uncertainty principle, in the case that only $I_P$ and $I_{P'}$ are disjoint, we need to deal with tails. The precise estimate we need is as follows. Assume that $|I_P|\ge |I_{P'}|$. Then for every integer $\nu\ge 0$ we have that
\begin{equation}
\label{eqn:tileinteraction}
|\langle \phi_P,\phi_{P'}\rangle|\lesssim |I_P|^{-\frac12} |I_{P'}|^{\frac12} (1+2^{-k_P}\norm{c(I_P)-c(I_{P'})})^{-\nu},
\end{equation}
where the implicit constant depends only on $\nu$ and $\phi$. See \cite[Lemma 2.1]{Thi06} for the version of this estimate for one-dimensional wave packets. Similarly, we have
\begin{equation}
\label{eqn:tileinteraction2}
|\langle \psi^N_P,\psi^N_{P'}\rangle|\lesssim \|m\|^2_{\mathscr{M}^{\nu}} |I_P|^{-\frac12} |I_{P'}|^{\frac12} (1+2^{-k_P}\norm{c(I_P)-c(I_{P'})})^{-\nu},
\end{equation}
for every integer $\nu\ge 0$ provided that $N\not\in \omega_{P(0)}\cup \omega_{P'(0)}$. We prove \eqref{eqn:tileinteraction} and \eqref{eqn:tileinteraction2} in Section \ref{sec:ac_wavepacketest}.

\section{Proof of the mass lemma}\label{sec:ac_mass}

In this section we prove Lemma \ref{lemma:ac_mass}. The proof is in essence the same as in \cite[Prop. 3.1]{LT00}. Let $\mathcal{P}$ be a finite set of tiles and set $\mu=\mathcal{M}(\mathcal{P})$. We define the set of heavy tiles by
$$\mathcal{P}_{\text{heavy}}=\left\{P\in\mathcal{P}\,:\,\mathcal{M}(P)>\frac\mu 4\right\}$$
and accordingly $\mathcal{P}_{\text{light}}=\mathcal{P}\backslash\mathcal{P}_{\text{heavy}}$. Then \eqref{eqn:ac_masslemma1} is automatically satisfied. It remains to show \eqref{eqn:ac_masslemma2}. By the definition of mass \eqref{eqn:ac_massdefsingletile} we know that for every $P\in\mathcal{P}_{\text{heavy}}$ there exists a $P^\prime=P^\prime(P)\in\overline{\mathcal{P}}$ with $P^\prime\ge P$ such that
\begin{align}
\label{eqn:ac_masslemmapf1}
\int_{E_{P^\prime}} w_{P^\prime}(x) dx>\frac{\mu}{4}
\end{align}
Note that $P^\prime$ need not be in $\mathcal{P}$. Let $\mathcal{P}^\prime$ be the maximal elements in 
$$\{P^\prime(P)\,:\,P\in\mathcal{P}_{\text{heavy}}\}$$ with respect to the partial order $\le$ of tiles. Then $\mathcal{P}_{\text{heavy}}$ is a union of trees with tops in $\mathcal{P}^\prime$. Therefore it suffices to show
\begin{align}
\label{eqn:ac_masslemmapf2}
\sum_{P\in\mathcal{P}^\prime} |I_{P}| \le \frac C\mu
\end{align}
First we rewrite \eqref{eqn:ac_masslemmapf1} as
\begin{align}
\label{eqn:ac_masslemmapf3}
\sum_{j=0}^\infty\int_{\substack{E_{P}\cap (\dil{I_{P}}{2^j}\backslash \dil{I_{P}}{2^{j-1}})}} w_{P}(x) dx>C \mu\sum_{j=0}^\infty 2^{-j}.
\end{align}
where we adopt the temporary convention that $\dil{I_{P}}{2^{-1}} =\emptyset$ and for $j\ge 0$,
\[ \dil{I_P}{2^j} = \prod_{i=1}^n \Big[c(I_P)_i-2^{(k_P+j)\alpha_i-1}, c(I_P)_i+2^{(k_P+j)\alpha_i-1}\Big). \]
Thus, for every $P\in\mathcal{P}^\prime$ there exists a $j\ge 0$ such that
\begin{align}
\label{eqn:ac_masslemmapf4}
\int_{\substack{E_{P}\cap (\dil{I_{P}}{2^j}\backslash \dil{I_{P}}{2^{j-1}})}} \frac{dx}{\left(1+2^{-k_P}\norm{x-c(I_{P}})\right)^{\nu_1}} >C |I_{P}|
\mu 2^{-j}.
\end{align}
Note that for $x\in \dil{I_{P}}{2^j}\backslash \dil{I_{P}}{2^{j-1}}$ we have
\[1+2^{-k_P}\norm{x-c(I_{P})}\ge C2^j. \]
Using this we obtain from \eqref{eqn:ac_masslemmapf4},
\begin{align}
\label{eqn:ac_masslemmapf4a}
|I_{P}|<C\mu^{-1} |E_{P}\cap \dil{I_{P}}{2^j}| 2^{-(\nu_1-1)j}.
\end{align}
Summarizing, we have shown that for every $P\in\mathcal{P}'$ there exists $j\ge 0$ such that \eqref{eqn:ac_masslemmapf4a} holds.
This leads us to define for every $j\ge 0$, a set of tiles $\mathcal{P}_j$ by
$$\mathcal{P}_j=\{P\in\mathcal{P}^\prime\,:\,|I_{P}|<C\mu^{-1} |E_{P}\cap \dil{I_{P}}{2^j}| 2^{-j(\nu_1-1)}\}.$$
The estimate \eqref{eqn:ac_masslemmapf2} will follow by summing over $j$ if we can show that
\begin{align}
\label{eqn:masslemmapf5}
\sum_{P\in \mathcal{P}_j} |I_P| \le C 2^{-j} \mu^{-1}
\end{align}
for all $j\ge 0$. To show \eqref{eqn:masslemmapf5} we use a covering argument reminiscent of Vitali's covering lemma.
Fix $j\ge 0$. For every tile $P=I_P\times\omega_P$ we have an enlarged tile $\dil{I_{P}}{2^j} \times \omega_P$ (this is not a tile anymore).
We inductively choose $P_i\in \mathcal{P}_j$ such that $|I_{P_i}|$ is maximal among the $P\in\mathcal{P}_j\backslash \{P_0,\dots,P_{i-1}\}$ and the enlarged tile of $P_i$ is disjoint from the enlarged tiles of $P_0,\dots,P_{i-1}$. Since $\mathcal{P}_j$ is finite, this process terminates after finitely many steps, so that we have selected a subset $\mathcal{P}^\prime_j=\{P_0,P_1,\dots\}\subset \mathcal{P}_j$ of tiles whose enlarged tiles are pairwise disjoint. By construction, for every $P\in\mathcal{P}_j$ there exists a unique $P^\prime\in \mathcal{P}^\prime_j$ such that $|I_P|\le |I_{P^\prime}|$ and the enlarged tiles of $P$ and $P^\prime$ intersect. We call $P$ \emph{associated} with $P^\prime$.\\
Now the claim is that if two tiles $P,Q\in\mathcal{P}_j$ are associated with the same $P^\prime\in\mathcal{P}^\prime_j$, then $I_P$ and $I_{Q}$ are disjoint. To see this note that $\omega_P$ intersects $\omega_{P^\prime}$ by definition. Thus, since $|I_P|\le |I_{P^\prime}|$, we have $\omega_{P^\prime}\subset \omega_P$. The same holds for $Q$. Therefore we have $\omega_{P^\prime}\subset \omega_P\cap \omega_{Q}$. But $P,Q\in \mathcal{P}_j\subset \mathcal{P}'$ are disjoint tiles, so we must have $I_P\cap I_{Q}=\emptyset$. Moreover, all tiles $P$ associated with $P^\prime$ satisfy $I_P\subset \dil{I_{P'}}{2^{j+2}}$. Therefore we get
\begin{align*}
\sum_{P\in\mathcal{P}_j} |I_P| &= \sum_{P^\prime\in\mathcal{P}^\prime_j} \sum_{\substack{P\in\mathcal{P}_j\\\text{assoc. with}\,P^\prime}} |I_P|=\sum_{P^\prime\in\mathcal{P}_j^\prime} \left| \bigcup\displaylimits_{\substack{P\in\mathcal{P}_j\\\text{assoc. with}\,P^\prime}} I_P\right| \\
&\le \sum_{P^\prime\in \mathcal{P}^\prime_j} 2^{(j+2)|\alpha|} |I_{P^\prime}|\le C \mu^{-1} 2^{-j(\nu_1-|\alpha|-1)}\sum_{P^\prime\in\mathcal{P}^\prime_j} |E\cap N^{-1}(\omega_{P^\prime})\cap \dil{I_{P'}}{2^j}|\\
&\le C 2^{-j} \mu^{-1},
\end{align*}
using that $\nu_1>|\alpha|+1$. The penultimate inequality is a consequence of \eqref{eqn:ac_masslemmapf4a} and the last inequality follows, because the sets $N^{-1}(\omega_{P^\prime})\cap \dil{I_{P'}}{2^j}$ are disjoint and $|E|\le 1$. 

\section{Proof of the energy lemma}\label{sec:ac_energy}
In this section we prove Lemma \ref{lemma:ac_energy}. We adapt the argument of Lacey and Thiele \cite[Prop. 3.2]{LT00}.
The tree selection algorithm of Lacey and Thiele relies on the natural ordering of real numbers. In our situation this can be replaced by any functional on $\R^n$ that separates $\omega_{P(0)}$ from $\omega_{P(r)}$ for every tile $P\in\overline{\mathcal{P}}$ (this was already observed in \cite{PT03}). Let $i_0$ be such that $r_{i_0}=1$ (exists because $r\not=0$). Let us introduce the projection to the $i_0$th coordinate: $\pi_0:\R^n\to\R$, $x\mapsto x_{i_0}$. Then we have that \begin{equation}\label{eqn:ac_ordering}
\pi_0(\xi)<\pi_0(\eta)
\end{equation}
holds for every $\xi\in\omega_{P(0)}, \eta\in\omega_{P(r)}, P\in\overline{\mathcal{P}}$.\\
Let $\varepsilon=\mathcal{E}(\mathcal{P})$. For a $2$--tree $\mathbf{T}_2$ we define
\[ \Delta(\mathbf{T}_2) = \left(\frac1{|I_{\mathbf{T}_2}|} \sum_{P\in\mathbf{T}_2} |\langle f,\phi_P\rangle|^2\right)^{1/2}. \]
We will now describe an algorithm to choose the desired collection of trees $\mathcal{T}$ and also an auxiliary collection of $2$--trees $\mathcal{T}_2$:
\begin{enumerate}
\item[(1)] Initialize $\mathcal{T}:=\mathcal{T}_2:=\emptyset$ and $\mathcal{P}^{\mathrm{stock}}:=\mathcal{P}$. 
\item[(2)] Choose a $2$--tree $\mathbf{T}_2\subset\mathcal{P}^{\mathrm{stock}}$ such that 
\begin{enumerate}
\item[(a)] $\Delta(\mathbf{T}_2)\ge \varepsilon/2$, and
\item[(b)] $\pi_0(c(\omega_{\mathbf{T}_2}))$ is minimal among all the $2$--trees in $\mathcal{P}^{\mathrm{stock}}$ satisfying (a).
\end{enumerate}
If no such $\mathbf{T}_2$ exists, then terminate.
\item[(3)] Let $\mathbf{T}$ be the maximal tree in $\mathcal{P}^{\mathrm{stock}}$ with top $P_{\mathbf{T}_2}$ (with respect to set inclusion).
\item[(4)] Add $\mathbf{T}$ to $\mathcal{T}$ and $\mathbf{T}_2$ to $\mathcal{T}_2$. Also, remove all the elements of $\mathbf{T}$ from $\mathcal{P}^{\mathrm{stock}}$. Then continue again with Step (2).
\end{enumerate}

Since $\mathcal{P}$ is finite it is clear that the algorithm terminates after finitely many steps. Also note for every $\mathbf{T}\in\mathcal{T}$ there exists a unique $\mathbf{T}_2\in\mathcal{T}_2$ with $\mathbf{T}_2\subset\mathbf{T}$, and vice versa. After the algorithm terminates we set $\mathcal{P}_{\mathrm{low}}=\mathcal{P}^{\mathrm{stock}}$ and $\mathcal{P}_{\mathrm{high}}$ to be the union of the trees in $\mathcal{T}$. Then, \eqref{eqn:ac_energylemma1} is automatically satisfied and it only remains to show
\begin{equation}\label{eqn:ac_energylemma3}
\sum_{\mathbf{T}_2\in\mathcal{T}_2} |I_{\mathbf{T}_2}| \lesssim \varepsilon^{-2}.
\end{equation}
Before we do that we establish a geometric property of the selected trees that will be crucial in the following. 
\begin{lemma}\label{lemma:ac_disjoint}
Let $\mathbf{T}_2\not=\mathbf{T}'_2\in \mathcal{T}_2$ and $P\in\mathbf{T}_2, P'\in\mathbf{T}'_2$. If $\omega_P\subset\omega_{1P'}$, then $I_{P'}\cap I_{\mathbf{T}_2}=\emptyset$.
\end{lemma}
\begin{proof}
Note that $c(\omega_{\mathbf{T}_2})\in\omega_P\subset\omega_{P'(0)}$ while $c(\omega_{\mathbf{T_2'}})\in\omega_{P'(r)}$. By \eqref{eqn:ac_ordering} and condition (b) in Step (2) we therefore conclude that $\mathbf{T}_2$ was chosen before $\mathbf{T}_2'$ during the above algorithm. Let $\mathbf{T}$ be the tree in $\mathcal{T}$ such that $\mathbf{T}_2\subset\mathbf{T}$. Thus, if $I_{P'}$ was not disjoint from $I_{\mathbf{T}_2}=I_\mathbf{T}$, then it would be contained in $I_\mathbf{T}$ and therefore $P'\le P_\mathbf{T}$ which means it would have been included into $\mathbf{T}$ during Step (3). That is a contradiction.
\end{proof}
The sum in \eqref{eqn:ac_energylemma3} equals
\[ \sum_{T_2\in\mathcal{T}_2} \Delta(\mathbf{T}_2)^{-2} \sum_{P\in\mathbf{T}_2} |\langle f,\phi_P\rangle| \le 4 \varepsilon^{-2} \sum_{P\in\bigcup\mathcal{T}_2} |\langle f,\phi_P\rangle|^2, \]
where $\bigcup\mathcal{T}_2=\bigcup_{\mathbf{T}_2\in\mathcal{T}_2} \mathbf{T}_2$. Let us write
\begin{equation}\label{eqn:ac_energylemma4}
\sum_{P\in\bigcup\mathcal{T}_2} |\langle f,\phi_P\rangle|^2=\Big\langle\sum_{P\in\bigcup\mathcal{T}_2}\langle f,\phi_P\rangle \phi_P, f\Big\rangle
\end{equation}
and use the Cauchy-Schwarz inequality to estimate this by
\begin{equation}\label{eqn:ac_energylemma5}
\Big\|\sum_{P\in\bigcup\mathcal{T}_2}\langle f,\phi_P\rangle \phi_P\Big\|_2, 
\end{equation}
where we used that $\|f\|_2=1$. So far we have shown that
\begin{equation}\label{eqn:ac_energylemma6}
\varepsilon^2 \sum_{\mathbf{T}_2\in\mathcal{T}_2} |I_{\mathbf{T}_2}| \lesssim \Big\|\sum_{P\in\bigcup\mathcal{T}_2}\langle f,\phi_P\rangle \phi_P\Big\|_2.
\end{equation}
Thus if we can show that
\begin{equation}\label{eqn:ac_energylemma7}
\Big\|\sum_{P\in\bigcup\mathcal{T}_2}\langle f,\phi_P\rangle \phi_P\Big\|_2^2 \lesssim \varepsilon^2 \sum_{\mathbf{T}_2\in\mathcal{T}_2} |I_{\mathbf{T}_2}|,
\end{equation}
then \eqref{eqn:ac_energylemma3} follows. Expanding the $L^2$ norm in \eqref{eqn:ac_energylemma7} we get that the left hand side is bounded by
\begin{equation}\label{eqn:ac_energylemma8}
\sum_{\substack{P,P'\in\bigcup\mathcal{T}_2,\\\omega_P=\omega_{P'}}} |\langle f,\phi_P\rangle \langle f,\phi_{P'}\rangle \langle \phi_P,\phi_{P'}\rangle| + 2 \sum_{\substack{P,P'\in\bigcup\mathcal{T}_2,\\\omega_P\subset\omega_{P'(0)}}} |\langle f,\phi_P\rangle \langle f,\phi_{P'}\rangle \langle \phi_P,\phi_{P'}\rangle|.
\end{equation}
Here we have used that $\langle \phi_P,\phi_P'\rangle=0$ if $\omega_{P(0)}\cap\omega_{P'(0)}=\emptyset$ and therefore either $\omega_P=\omega_{P'}$, $\omega_P\subset\omega_{P'(0)}$, or $\omega_{P'}\subset\omega_{P(0)}$ (the last two cases are symmetric). We treat both sums in this term separately. Estimating the smaller one of $|\langle f,\phi_P\rangle|$ and $|\langle f,\phi_{P'}\rangle$ by the larger one, we obtain that the first sum in \eqref{eqn:ac_energylemma8} is
\[\lesssim \sum_{P\in\bigcup\mathcal{T}_2} |\langle f,\phi_P\rangle|^2 \sum_{\substack{P'\in\bigcup\mathcal{T}_2,\\\omega_P=\omega_{P'}}} |\langle \phi_P,\phi_{P'}\rangle|. \]
Using \eqref{eqn:tileinteraction} we estimate this by
\begin{equation}\label{eqn:ac_energylemma9}
\sum_{P\in\bigcup\mathcal{T}_2} |\langle f,\phi_P\rangle|^2 \sum_{\substack{P'\in\bigcup\mathcal{T}_2,\\\omega_P=\omega_{P'}}}  (1+2^{-k_P}\norm{c(I_P)-c(I_{P'})})^{-\nu}.
\end{equation}
Notice that $I_P\cap I_{P'}=\emptyset$ for $P\not=P'$ in the inner sum. This implies 
\[ \sum_{\substack{P'\in\bigcup\mathcal{T}_2,\\\omega_P=\omega_{P'}}}  (1+2^{-k_P}\norm{c(I_P)-c(I_{P'})})^{-\nu}\lesssim \int_{\R^n} (1+\norm{x})^{-\nu} dx \lesssim 1, \]
provided that $\nu>|\alpha|$. Therefore \eqref{eqn:ac_energylemma9} is
\begin{equation}
\lesssim \sum_{\mathbf{T}_2\in\mathcal{T}_2} \sum_{P\in \mathbf{T}_2} |\langle f,\phi_P\rangle|^2 \le \varepsilon^2 \sum_{\mathbf{T}_2\in\mathcal{T}_2} |I_{\mathbf{T}_2}|,
\end{equation}
as desired. It remains to estimate the second sum in \eqref{eqn:ac_energylemma8}. To that end it suffices to show that
\begin{equation}\label{eqn:ac_energylemma10} \sum_{P\in\mathbf{T}_2} \sum_{P'\in\mathcal{S}_P} |\langle f,\phi_P\rangle \langle f,\phi_{P'}\rangle \langle \phi_P,\phi_{P'}\rangle| \lesssim \varepsilon^2 |I_{\mathbf{T}_2}|,
\end{equation}
for every $\mathbf{T}_2\in\mathcal{T}_2$, where
\[\mathcal{S}_P=\Big\{P'\in\bigcup\mathcal{T}_2\,:\,\omega_P\subset\omega_{P'(0)}\Big\}. \]
Here we follow the argument given in \cite{Lac04}. Observe that if $P\in\mathbf{T}_2$, then $\mathcal{S}_P\cap \mathbf{T}_2=\emptyset$.
Interpreting the singleton
$\{P\}$ as a $2$--tree we obtain
\begin{equation}\label{eqn:ac_singleton}
|\langle f,\phi_P\rangle|\le \varepsilon |I_P|^{1/2}
\end{equation}
for all $P\in\mathcal{P}$. Combining this with \eqref{eqn:tileinteraction} we can estimate the left hand side of \eqref{eqn:ac_energylemma10} by
\begin{equation}\label{eqn:ac_energylemma11}
\varepsilon^2 \sum_{P\in\mathbf{T}_2} \sum_{P'\in\mathcal{S}_P} |I_{P'}| (1+2^{-k_P}\norm{c(I_P)-c(I_{P'})})^{-\nu}. 
\end{equation}
Indeed, Lemma \ref{lemma:ac_disjoint} implies that $I_{\mathbf{T}_2}\cap I_{P'}=\emptyset$ for every $P\in\mathbf{T}_2, P'\in\mathcal{S}_P$. Moreover, it also implies that for $P'\not=P''\in\mathcal{S}_P$ we have $I_{P'}\cap I_{P''}=\emptyset$. These facts facilitate the following estimate:
\[ \sum_{P\in\mathbf{T}_2} \sum_{P'\in\mathcal{S}_P} |I_{P'}| (1+2^{-k_P}\norm{c(I_P)-c(I_{P'})})^{-\nu} \lesssim \sum_{P\in\mathbf{T}_2} \sum_{P'\in\mathcal{S}_P} \int_{I_{P'}} (1+2^{-k_P}\norm{c(I_P)-x})^{-\nu} dx \]
\[\lesssim \sum_{P\in\mathbf{T}_2} \int_{(I_{\mathbf{T}_2})^c} (1+2^{-k_P}\norm{c(I_P)-x})^{-\nu}. \]
Since $\mathbf{T}_2$ is a tree, the last quantity can be estimated by
\[
\sum_{k\le k_{\mathbf{T}_2}} \sum_{u\in Q_k\cap (\Z^n+\frac12)}  \int_{(I_{\mathbf{T}_2})^c} (1+\norm{u-\dil{x}{2^{-k}}})^{-\nu} dx,
\]
where $Q_k\in\mathcal{D}^\alpha$ is an anisotropic dyadic rectangle of scale $k_{\mathbf{T}_2}-k$ that is given by a rescaling of $I_{\mathbf{T}_2}$. The previous display is no greater than a constant times
\begin{align}\label{eqn:ac_energyestpfend}
\sum_{k\le k_{\mathbf{T}_2}} 2^{k|\alpha|} \left(\sum_{u\in Q_k\cap (\Z^n+\frac12)}  (1+\mathrm{dist}_\alpha((Q_k)^c, u)^{-|\alpha|-\gamma}\right) \left(\int_{\R^n} (1+\norm{x})^{-(\nu-|\alpha|-\gamma)} dx\right),
\end{align}
where $\nu>2|\alpha|$ and $\gamma$ is a fixed and sufficiently small positive constant. The integral over $x$ in the previous display is bounded by a constant depending on $\nu-|\alpha|-\gamma>|\alpha|$. To estimate the sum over $u$ we note that for every $u$ in the indicated range there exists a lattice point $v\in \partial Q_k\cap \Z^n$ such that $\mathrm{dist}_\alpha((Q_k)^c,u)\ge \frac12 \norm{v-u}$.
Thus we may bound the sum over $u$ by
\[ \sum_{v\in \partial Q_k\cap \Z^n} \sum_{u\in \Z^n+\frac12}  (1+\norm{v-u})^{-|\alpha|-\gamma}\lesssim |\partial Q_k\cap \Z^n|\lesssim 2^{(k_{\mathbf{T}_2}-k)|\alpha|_\infty}. \]
Thus, \eqref{eqn:ac_energyestpfend} is bounded by a constant times
\[ 2^{k_{\mathbf{T}_2}|\alpha|_\infty} \sum_{k\le k_{\mathbf{T}_2}} 2^{k(|\alpha|-|\alpha|_\infty)} \lesssim 2^{k_{\mathbf{T}_2}|\alpha|}=|I_{\mathbf{T}_2}|. \]
This proves \eqref{eqn:ac_energylemma10}.

\section{Proof of the tree estimate}\label{sec:ac_tree}

In this section we prove Lemma \ref{lemma:ac_tree}. This is the core of the proof. 
For a rectangle $I=\prod_{i=1}^n I_i\in\mathcal{D}^\alpha$ we denote by $\fam{I}$ the enlarged rectangle defined by
\[ \fam{I} = \prod_{i=1}^n (2^{\alpha_i+1}-1) I_i. \]
Here $\lambda I_i$ is the interval of length $\lambda |I_i|$ with the same center as $I_i$. Let $\mathcal{J}$ be the partition of $\R^n$ that is given by the collection of maximal
anisotropic dyadic rectangles $J\in\mathcal{D}^\alpha$ such that $\fam{J}$ does not contain any $I_P$ with $P\in \mathbf{T}$ (maximal with respect to inclusion). 
Set $\varepsilon=\mathcal{E}(\mathbf{T})$ and $\mu=\mathcal{M}(\mathbf{T})$. 
Choose phase factors $(\epsilon_P)_P$ of modulus $1$ such that
\[\sum_{P\in \mathbf{T}} |\langle f,\phi_P\rangle\langle\psi_P^{N(\cdot)}, \mathbf{1}_{E_{P(r)}}
\rangle|=\int_{\R^n}\sum_{P\in \mathbf{T}} \epsilon_P \langle f,\phi_P\rangle \psi_P^{N(x)}(x) \mathbf{1}_{E_{P(r)}}(x) dx\]
\[ \le \left\|\sum_{P\in \mathbf{T}} \epsilon_P \langle f,\phi_P\rangle
\psi_P^N \mathbf{1}_{E_{P(r)}}\right\|_{1}
\le \mathcal{K}_1 + \mathcal{K}_2,
\]
where
\[\mathcal{K}_1 = \sum_{J\in\mathcal{J}} \sum_{P\in \mathbf{T}, |I_P|\le |\parent{J}|}\|\langle f,\phi_P\rangle \psi_P^{N(\cdot)} \mathbf{1}_{E_{P(r)}} \|_{L^1(J)},\]
\[\mathcal{K}_2 = \sum_{J\in\mathcal{J}} \left\|\sum_{P\in \mathbf{T}, |I_P|>|\parent{J}|}
\epsilon_P\langle f,\phi_P\rangle \psi_P^{N(\cdot)} \mathbf{1}_{E_{P(r)}} \right\|_{L^1(J)}.\]
We first estimate $\mathcal{K}_1$. This is the easy part, since in the sum defining $\mathcal{K}_1$ we have that $I_P$ is disjoint from $\fam{J}$. Again, interpreting the singleton $\{P\}$ as a $2$--tree we see that \eqref{eqn:ac_singleton} holds for all $P\in\mathbf{T}$. This gives
\[
\mathcal{K}_1 \le \varepsilon \sum_{J\in\mathcal{J}} \sum_{\substack{P\in\mathbf{T}\\|I_P|\le |\parent{J}|}} 2^{|\alpha|k_P/2} \| \psi_P^{N(\cdot)} \mathbf{1}_{E_{P(r)}} \|_{L^1(J)}.
\]
Using \eqref{eqn:psiestimate} the previous display is seen to be no larger than a constant times
\[ \|m\|_{\mathscr{M}^{\nu_0}} \varepsilon \sum_{J\in\mathcal{J}} \sum_{\substack{P\in\mathbf{T}\\|I_P|\le |\parent{J}|}} \int_{J\cap E_{P(r)}} w^{\nu_0}(2^{-k_P}(x-c(I_P))) dx \]
\begin{equation}\label{eqn:ac_K1est1} \le \|m\|_{\mathscr{M}^{\nu_0}} \varepsilon \mu  \sum_{J\in\mathcal{J}} \sum_{\substack{P\in\mathbf{T}\\|I_P|\le |\parent{J}|}} 2^{|\alpha|k_P} \sup_{x\in J} w^{2|\alpha|+\frac23}(2^{-k_P}(x-c(I_P))),
\end{equation}
where we have set $\nu_1=|\alpha|+\frac43$.
Since $I_P$ is disjoint from $\fam{J}$ we can estimate \eqref{eqn:ac_K1est1} as 
\begin{equation}\label{eqn:ac_K1est2} \lesssim  \|m\|_{\mathscr{M}^{\nu_0}} \varepsilon\mu  \sum_{J\in\mathcal{J}} \sum_{\substack{k\in\Z,\\2^{k|\alpha|}\le |\parent{J}|}}2^{|\alpha|k} \sum_{\substack{P\in\mathbf{T},\\k_P=k}}  
w^{2|\alpha|+\frac23}(2^{-k}\mathrm{dist}_\alpha(J,I_P)).
\end{equation}
Before we proceed, we claim that for every $\nu>|\alpha|$, $k\in\Z$ and fixed $J\in\mathcal{J}$ with $2^{k|\alpha|}\le |\parent{J}|$ we have
\begin{equation}\label{eqn:ac_K1est3} \sum_{\substack{P\in\mathbf{T},\\k_P=k}}  
w^{\nu}(2^{-k}\mathrm{dist}_\alpha(J,I_P)) \lesssim 1,
\end{equation}
where the implicit constant blows up as $\nu$ approaches $|\alpha|$. To verify the claim, let us assume for simpler notation that $J$ is centered at the origin. Then by disjointness of $I_P$ and $\fam{J}$ we have \[\mathrm{dist}_\alpha(J,I_P)\gtrsim \mathrm{dist}_\alpha(0,I_P)\gtrsim 2^k \rho(m),\]
where $m=(m_1,\dots,m_n)\in\Z^n$ is such that $I_P=\prod_{i=1}^n [2^{k\alpha_i} m_i, 2^{k\alpha_i} (m_i+1))$.
Thus the sum in \eqref{eqn:ac_K1est3} is
\[ \lesssim \sum_{m\in\Z^n} (1+\rho(m))^{-\nu}, \]
which implies the claim.

Estimate \eqref{eqn:ac_K1est2} by
\begin{equation}\label{eqn:ac_K1est4}
\lesssim  \|m\|_{\mathscr{M}^{\nu_0}} \varepsilon\mu  \sum_{J\in\mathcal{J}}w^{|\alpha|+\frac13}(2^{-k_\mathbf{T}}\mathrm{dist}_\alpha(J,I_T)) \sum_{\substack{k\in\Z,\\2^{k|\alpha|}\le |\parent{J}|}}2^{|\alpha|k},
\end{equation}
Here $k_\mathbf{T}$ is the scale of $I_\mathbf{T}$ and we have used \eqref{eqn:ac_K1est3} and
\[2^{-k} \mathrm{dist}_\alpha(J,I_P)\ge 2^{-k_\mathbf{T}} \mathrm{dist}_\alpha (J,I_\mathbf{T}). \]
Summing the geometric series, \eqref{eqn:ac_K1est4} is
\[\lesssim \|m\|_{\mathscr{M}^{\nu_0}} \varepsilon\mu  \sum_{J\in\mathcal{J}} w^{|\alpha|+\frac13}(2^{-k_\mathbf{T}}\mathrm{dist}_\alpha(J,I_P)) |J|. \]
The sum in that expression can be estimated as follows:
\[ \sum_{J\in\mathcal{J}} w^{|\alpha|+\frac13}(2^{-k_\mathbf{T}}\mathrm{dist}_\alpha(J,I_P)) |J| \lesssim \sum_{J\in\mathcal{J}} \int_J (1+2^{-k_\mathbf{T}}\rho(x-c(I_\mathbf{T})))^{-(|\alpha|+\frac13)} dx.\]
By disjointness of the $J$ we can bound this by
\[ \int_{\R^n} (1+2^{-k_\mathbf{T}}\rho(x-c(I_\mathbf{T})))^{-(|\alpha|+\frac13)} dx = |I_\mathbf{T}| \int_{\R^n} (1+\rho(x))^{-(|\alpha|+\frac13)} dx\lesssim |I_\mathbf{T}|. \]
To summarize, we showed that
\begin{equation}\label{eqn:tailest}
\mathcal{K}_1\lesssim \|m\|_{\mathscr{M}^{\nu_0}} \varepsilon \mu |I_\mathbf{T}|,
\end{equation}
using that $\nu_0\ge 3|\alpha|+2$.\\
Let us proceed to estimating $\mathcal{K}_2$. This is more difficult. We may assume that the sum runs only over
those $J$ for which there is a $P\in \mathbf{T}$ such that $|I_P|>|\parent{J}|$. Then
$|I_\mathbf{T}|>|\parent{J}|$ and $J\subset \fam{I_\mathbf{T}}$.
From now on let such a $J$ be fixed. Define
\begin{align}
\label{eqn:GJdef}
G_{J}=J\cap \bigcup_{P\in \mathbf{T}, |I_{P}|>|\parent{J}|} E_{P(r)}
\end{align}
Before proceeding we prove the following.
\begin{lemma}\label{lemma:ac_GJest}
There exists a constant $C>0$ independent of $J$ such that
\begin{align}
\label{eqn:GJest}
|G_{J}|\le C\mu |J|
\end{align}
\end{lemma}

\begin{proof}
By definition of $J$, there exists $P_0\in\mathbf{T}$ such that $I_{P_0}$ is contained in $\fam{\parent{J}}$. 
We claim that there exists a tile $P_0<P'<P_\mathbf{T}$ such that $|I_{P'}|=|J^{++}|$.
Indeed, note $|I_{P_0}|\le |J^{++}|$. If there is equality, we simply take $P'=P_0$. Otherwise we take $I_{P'}\in\mathcal{D}^\alpha$ to be the unique dyadic ancestor of $I_{P_0}$ such that $|I_{P'}|=|J^{++}|$ and choose $\omega_{P'}$ accordingly such that it contains $c(\omega_\mathbf{T})$.
Now we have
\[|\omega_P|=|I_P|^{-1}\le |\gparent{J}|^{-1}= |I_{P'}|^{-1} = |\omega_{P'}|\]
for every tile $P\in\mathbf{T}$ with $|I_P|>|J^+|$. This implies $\omega_P\subset \omega_{P'}$ and thus
\[ G_J \subset J\cap E_{P'}. \]
As a consequence,
\[ |G_J| \le \int_{E_{P'}} \mathbf{1}_J(x) dx \lesssim |I_{P'}| \int_{E_{P'}} w_{P'}(x) dx\lesssim \mu |J|. \]
\end{proof}
Let us define
\begin{align}
\label{eqn:FJdef}
F_{J}=\sum_{\substack{P\in \mathbf{T},\\|I_{P}|>|\parent{J}|}} \epsilon_{P} \langle f,\phi_{P}\rangle \psi^{N(\cdot)}_{P} \mathbf{1}_{E_{P(r)}}.
\end{align}
Since every tree can be written as the union of a $1$--tree and a $2$--tree, we may treat each of these cases separately. 

\subsection{The case of $1$--trees}
Assume that $\mathbf{T}$ is a $1$--tree. This is the easier case. The reason is that for every $P,P'\in\mathbf{T}$, $\omega_P\not=\omega_{P'}$ we have that $\omega_{P(r)}$ and $\omega_{P'(r)}$ are disjoint and thus we have good orthogonality of the summands in \eqref{eqn:FJdef}.
Using \eqref{eqn:ac_singleton} and \eqref{eqn:psiestimate} we see that
\[|F_J(x)| \lesssim \|m\|_{\mathscr{M}^{\nu_0}} \varepsilon \sum_{\substack{P\in \mathbf{T},\\|I_{P}|>|\parent{J}|}} (1+2^{-k_P}\norm{x-c(I_P)})^{-\nu_0} \mathbf{1}_{E_{P(r)}}(x). \]
Using disjointness of the $E_{P(r)}$ this can be estimated by
\[ \|m\|_{\mathscr{M}^{\nu_0}} \varepsilon\cdot \sup_{k\in\Z} \sum_{m\in\Z^n+\frac12} (1+2^{-k}\norm{x-\dil{m}{2^k}})^{-\nu_0}. \]
By an index shift we see that
\[ \sum_{m\in\Z^n+\frac12} (1+2^{-k}\norm{x-\dil{m}{2^k}})^{-\nu_0} = \sum_{m\in\Z^n+\frac12} (1+\norm{m+\gamma})^{-\nu_0},\]
where $\gamma\in [0,1]^n$ depends on $k$ and $x$. The last sum is $\lesssim 1$ independently of $\gamma$. Thus we proved the pointwise estimate
\begin{equation}
|F_{J}(x)|\lesssim \|m\|_{\mathscr{M}^{\nu_0}} \varepsilon.
\end{equation}
Combining this with the support estimate \eqref{eqn:GJest} we obtain
\begin{align}
\|F_{J}\|_{L^{1}(J)}\lesssim \|m\|_{\mathscr{M}^{\nu_0}} \varepsilon\mu |J|.
\end{align}
Summing over the pairwise disjoint $J\subset \fam{I_\mathbf{T}}$ we obtain
\[ \mathcal{K}_2 \lesssim \|m\|_{\mathscr{M}^{\nu_0}} \varepsilon\mu |I_\mathbf{T}| \]
as desired. Note that we only needed $\nu_0>|\alpha|$ to obtain this estimate.

\subsection{The case of $2$--trees}
Here we assume that $\mathbf{T}$ is a $2$--tree. The additional $x$-dependence present in the wave packets $\psi^{N(x)}_P$ makes this part more difficult than the congruent argument in \cite{LT00}. This problem arises already in the isotropic case \cite{PT03}. 
The goal is again to obtain a pointwise estimate for $F_J$. In the following we fix $x\in J $ such that $F_{J}(x)\not=0$. Observe that the $\omega_{P(r)}$, $P\in\mathbf{T}$ are nested. Let us denote the smallest (resp. largest) $\omega_{P(r)}$ (resp. $\omega_P$) such that $x\in N^{-1}(\omega_{P(r)})\cap E$ by $\omega_-$ (resp. $\omega_+$).  Let $k_+\in\Z$ be such that $|\omega_+|=2^{k_+|\alpha|}$ and $k_-\in\Z$ such that $|\omega_-|=2^{-k_-|\alpha|-n}$ (note from the definition that $\omega_-\not\in\mathcal{D}^\alpha$ if $\alpha\not=(1,\dots,1)$).
Then the nestedness property implies
\[ F_J(x) = \sum_{\substack{P\in\mathbf{T},\\ k_+ \le k_P\le k_-}}\epsilon_P \langle f,\phi_P\rangle \psi^{N(x)}_P(x) \]
Define
\[ h_x = \mathrm{M}_{c(\omega_+)}\mathrm{D}^1_{2^{k_+}}\phi_+ - M_{c(\omega_-)} D^1_{ 2^{k_-}}\phi_-, \]
where $\phi_+(x) = b_1^{-n} \phi(b_1^{-1} x)$ and $\phi_-$ is a Schwartz function satisfying $0\le \widehat{\phi_-}\le 1$ such that $\widehat{\phi_-}$ is supported on $[-\frac{b_2}2,\frac{b_2}2]$ and equals to one on $[-\frac{b_3}2,\frac{b_3}2]$, where $b_{j+2}=\frac12+b_j$ for $j=0,1$.
From the definition we see that $\widehat{h_x}$ is supported on $b_0 b_1^{-1}\omega_+ \cap (2b_3 \omega_-)^c$ and equal to one on $\omega_+\cap (2b_2\omega_-)^c$. In particular, $\widehat{h_x}(\xi)$ equals to one if $\xi\in \mathrm{supp}\,\phi_P$ and $k_+\le k_P\le k_-$ and vanishes if $k_P$ is outside this range.
For technical reasons that become clear further below we need the support of $\widehat{h_x}$ to keep a certain distance to $\omega_-$.
We obtain
\[ F_J(x) = \sum_{P\in\mathbf{T}} \epsilon_P \langle f,\phi_P\rangle (\psi^{N(x)}_P*h_x)(x). \]
Fix $\xi_0\in\omega_{\mathbf{T}}$. We decompose
\begin{align}
F_{J}(x)&=\sum_{P\in\mathbf{T}} \epsilon_P\langle f,\phi_P\rangle(\psi_P^{\xi_0}*h_x)(x) +
\sum_{P\in \mathbf{T}} \epsilon_P\langle f,\phi_P\rangle(\psi_P^{N(x)}-\psi_P^{\xi_0})*h_x)(x)\\\label{eqn:F2Jest2}
&=  G * \mathrm{M}_{\xi_0} K * h_x (x) + G*(\mathrm{M}_{N(x)} K-\mathrm{M}_{\xi_0} K)*h_x(x)
\end{align}
where
\begin{equation}
\label{eqn:G1def}
G=\sum_{P\in\mathbf{T}}\epsilon_P\langle f,\phi_P\rangle \phi_P.
\end{equation}
Before proceeding with the proof we record the following simple variant of a standard fact about maximal functions (see \cite{Duo01}).

\begin{lemma}\label{lemma:ac_max}
Let $\lambda>0$ and $w$ be an integrable function on $\R^n$ which is constant on $\{\rho(y)\le \lambda\}$ and radial and decreasing with respect to $\rho$, i.e.
\[w(x)\le w(y)\]
if $\norm{x}\ge \norm{y}$, with equality if $\norm{x}=\norm{y}$.
Let $x\in\R^n$ and $J\subset\R^n$ be such that $J\subset \{y\,:\,\norm{x-y}\le \lambda\}$.
Then we have
\[ |F*w|(x) \le \|w\|_1 \sup_{J\subset I} \frac1{|I|} \int_I |F(y)| dy, \]
where the supremum is taken over all anisotropic cubes $I\subset\R^n$.
\end{lemma}

\begin{proof}
First we assume that $w$ is a step function. That is,
\[ w(y) = \sum_{j=1}^\infty c_j \mathbf{1}_{\norm{y}\le r_j} \]
with $\lambda\le r_1<r_2<\cdots$.
Then we have
\[ F* w (x) = \sum_j r_j^{|\alpha|} c_j \frac1{r_j^{|\alpha|}}\int_{\norm{x-y}\le r_j} |F(y)| dy \le \|w\|_1 \sup_{J\subset I} \frac1{|I|} \int_I |F(y)| dy. \]
The general case follows by approximation of $w$ by step functions and an application of Lebesgue's dominated convergence theorem.
\end{proof}

Since
\begin{equation}\label{eqn:ac_hxest}
|h_x(y)|\lesssim 2^{-k_+|\alpha|} |\phi_+|(\dil{y}{2^{-k_+}})+2^{-k_-|\alpha|} |\phi_-|(\dil{y}{2^{-k_-}}) 
\end{equation}
and $x\in J$, $|J|\le 2^{k_+|\alpha|}\le 2^{k_-|\alpha|}$ we have from Lemma \ref{lemma:ac_max} that
\begin{equation}\label{eqn:ac_maximalest1}
|G*\mathrm{M}_{\xi_0}K*h_x(x)|\lesssim \sup_{J\subset I}\frac1{|I|}\int_I |G*\mathrm{M}_{\xi_0}K(y)| dy.
\end{equation}
Let us assume for the moment that we also have the estimate
\begin{equation}\label{eqn:ac_errest}
|G*(\mathrm{M}_{N(x)} K-\mathrm{M}_{\xi_0} K)*h_x(x)| \lesssim \|m\|_{\mathscr{M}^{\nu_0}}\sup_{J\subset I} \frac1{|I|}\int_I |G(y)| dy.
\end{equation}
We will first show how to finish the proof from here. At the end of the section we will then show that \eqref{eqn:ac_errest} indeed holds.\\
From \eqref{eqn:F2Jest2}, \eqref{eqn:ac_maximalest1}, \eqref{eqn:ac_errest} and Lemma \ref{lemma:ac_GJest} we see that
\[ \sum_{\substack{J\in\mathcal{J},\\J\subset \fam{I_\mathbf{T}}}} \|F_J\|_{L^1(J)} \lesssim \mu \sum_{\substack{J\in\mathcal{J},\\J\subset \fam{I_\mathbf{T}}}} |J| \left(\sup_{J\subset I}\frac1{|I|}\int_I |G*\mathrm{M}_{\xi_0}K(y)| dy + \|m\|_{\mathscr{M}^{\nu_0}}\sup_{J\subset I} \frac1{|I|}\int_I |G(y)| dy\right) \]
By disjointness of the $J\in\mathcal{J}$ this is no greater than
\begin{equation}\label{eqn:ac_twotreeest1}
\mu \left(\left\|\mathcal{M}(G*\mathrm{M}_{\xi_0}K)\right\|_{L^1(\fam{I_\mathbf{T}})} +\|m\|_{\mathscr{M}^{\nu_0}}\left\| \mathcal{M}(G) \right\|_{L^1(\fam{I_\mathbf{T}})}\right),
\end{equation}
where $\mathcal{M}$ denotes the maximal function defined by
\[ \mathcal{M} F(y) = \sup_{y\in I} \frac1{|I|} \int_I |F|,\]
where the supremum runs over all anisotropic cubes $I\subset\R^n$.
Clearly, $\mathcal{M}$ is a bounded operator $L^2(\R^n)\to L^2(\R^n)$, because it is bounded pointwise by a composition of one-dimensional Hardy-Littlewood maximal functions applied in each component.

Applying the Cauchy-Schwarz inequality and the $L^2$ boundedness of $\mathcal{M}$ we see that \eqref{eqn:ac_twotreeest1} is 
\[ \lesssim \mu |I_T|^{\frac12} \left(\left\|G*\mathrm{M}_{\xi_0}K\right\|_2 + \|m\|_{\mathscr{M}^{\nu_0}}\left\|G\right\|_2\right). \]
By repeating the  arguments that lead to the proof of \eqref{eqn:ac_energylemma7}, using \eqref{eqn:tileinteraction} or \eqref{eqn:tileinteraction2}, respectively, we obtain that
\[ \left\|G*\mathrm{M}_{\xi_0}K\right\|_2 +  \|m\|_{\mathscr{M}^{\nu_0}}\left\|G\right\|_2\lesssim \|m\|_{\mathscr{M}^{\nu_0}}\varepsilon  |I_\mathbf{T}|^{\frac12}.\]
This concludes the proof. It remains to prove \eqref{eqn:ac_errest}. Let us write 
\[ R(y)=(\mathrm{M}_{N(x)}K-\mathrm{M}_{\xi_0}K)*h_x(y). \]
We will give two different estimates for $R$. The first one is only effective if $\norm{y}$ is large and the second one if $\norm{y}$ is small. Let us start with the first estimate. By Fourier inversion, we can write $R(y)$ (up to a constant) as
\begin{equation}\label{eqn:ac_errestpfpt1_1}
\int_{\R^n} (m(\xi-N(x))-m(\xi-\xi_0)) \widehat{h_x}(\xi) e^{i\xi y} d\xi.
\end{equation}
Fix $y$ and let $i$ be such that $\norm{y}=|y_i|^{1/\alpha_i}$. Then we integrate by parts in the $i$th component to see that \eqref{eqn:ac_errestpfpt1_1} is bounded by
\begin{equation}\label{eqn:ac_errestpfpt1_2}
\lesssim \norm{y}^{-\nu'\alpha_i} \int_{\R^n}\Big| \partial^{\nu'}_{\xi_i} \Big[ (m(\xi-N(x))-m(\xi-\xi_0)) \widehat{h_x}(\xi) \Big] \Big| d\xi
\end{equation}
for integer $\nu'\ge 0$, where we have used that $\norm{y}\ge 2^{k_-}$ to estimate $|\dil{y}{2^{-k_-}}|\ge 2^{-k_-}\norm{y}$.

Let $\ell\le \nu_0$ be a non-negative integer. Using \eqref{eqn:mderivest} we obtain
\begin{equation}\label{eqn:ac_errestpfpt1_3}
\Big|\partial^\ell_{\xi_i}\Big[m(\xi-N(x))-m(\xi-\xi_0)\Big]\Big| \le \|m\|_{\mathscr{M}^\ell} ( \norm{\xi-N(x)}^{-\ell\alpha_i} + \norm{\xi-\xi_0}^{-\ell\alpha_i}).
\end{equation}
Recall that $\xi_0$ and $N(x)$ are contained in $\omega_-$ and the integrand of \eqref{eqn:ac_errestpfpt1_2} is supported on $b_0 b_1^{-1}\omega_+ \cap (2b_3 \omega_-)^c$. Also, there exist $\omega_1,\dots,\omega_M\in\mathcal{D}^\alpha$ such that
\[\omega_-\subsetneq\omega_1\subsetneq \cdots\subsetneq \omega_M=\omega_+\]
and $|\omega_{j}|=2^{-k_j|\alpha|}$ with $k_1=k_-$ and $k_{j+1}=k_j-1$. If $\xi\in (2b_3\omega_-)^c$ we have 
\begin{equation}\label{eqn:ac_errestpfpt1_4}
\min(\norm{\xi-N(x)},\norm{\xi-\xi_0}) \gtrsim 2^{-k_-}.
\end{equation}
On the other hand, if $\xi\in (b_0b_1^{-1}\omega_j)\cap \omega_{j-1}^c$ for  $j=2,\dots,M$, then
\begin{equation}\label{eqn:ac_errestpfpt1_5}
\min(\norm{\xi-N(x)},\norm{\xi-\xi_0})\gtrsim 2^{-k_j}.
\end{equation}
Combining \eqref{eqn:ac_errestpfpt1_3} and \eqref{eqn:ac_errestpfpt1_4}, \eqref{eqn:ac_errestpfpt1_5} we get
\begin{equation}\label{eqn:ac_errestpfpt1_6}
|\partial^\ell_{\xi_i}\Big[m(\xi-N(x))-m(\xi-\xi_0)\Big]| \lesssim \|m\|_{\mathscr{M}^\ell} \sum_{j=1}^M 2^{k_j \ell\alpha_i} \mathbf{1}_{b_0b_1^{-1}\omega_j}(\xi).
\end{equation}
We also have
\begin{equation}\label{eqn:ac_errestpfpt1_7}
|\partial^\ell_{\xi_i}\widehat{h_x}(\xi)| \lesssim 2^{k_+\ell\alpha_i}\mathbf{1}_{b_0b_1^{-1}\omega_+}(\xi) +2^{k_-\ell\alpha_i} \mathbf{1}_{2b_3\omega_-}(\xi).
\end{equation}

Thus we see from \eqref{eqn:ac_errestpfpt1_6} and \eqref{eqn:ac_errestpfpt1_7} that for all $i=1,\dots,n$ and $0\le\ell\le \nu'$ we obtain
\[ \int_{\R^n} \Big| \partial^{\ell}_{\xi_i} \Big[ m(\xi-N(x))-m(\xi-\xi_0) \Big] \partial^{\nu'-\ell}_{\xi_i}  \widehat{h_x}(\xi)  \Big|d\xi \lesssim \|m\|_{\mathscr{M}^{\nu'}} 2^{k_-(\nu'\alpha_i-|\alpha|)}, \] 
provided that $\nu'\alpha_i\ge |\alpha|$. Setting $\nu'=\lceil \frac{\nu_0}{\alpha_i}\rceil$, we have shown that
\begin{equation}\label{eqn:ac_errestpfpt1_8}
|R(y)| \lesssim \|m\|_{\mathscr{M}^{\nu_0}} 2^{-k_-|\alpha|} (2^{-k_-}\norm{y})^{-\nu_0}.
\end{equation}
It remains to find a good estimate for $R(y)$ when $\norm{y}$ is small. 
Let us estimate
\[ |R(y)|\le R_+(y) + R_-(y),\]
where
\[ R_{\pm} = |(\mathrm{M}_{N(x)}K-\mathrm{M}_{\xi_0}K)*\mathrm{D}^1_{2^{k_\pm}}\phi_{\pm}|.\]
The first claim is that if $\rho(y)\le 2^{k_\pm+1}$, then
\begin{equation}\label{eqn:ac_Rpmest}
R_{\pm}(y) \lesssim  \|m\|_{\mathscr{M}^{\nu_0}} 2^{-k_\pm|\alpha|}.
\end{equation}
(Here and throughout the proof of this claim $\pm$ always stands for a fixed choice of sign, either $+$ or $-$.) To see this, we first estimate $R_\pm(y)$ by
\[ 2^{-k_\pm|\alpha|} \int_{\R^n} |(e^{i(N(x)-\xi_0)z}-1) K(z) \phi_\pm(2^{-k_\pm}(y-z))| dz \lesssim 2^{-k_\pm|\alpha|}(\mathbf{I} + \mathbf{II}),\]
where 
\[ \mathbf{I} =  \int_{\norm{z}\le 2^{k_\pm+2}} |(e^{i(N(x)-\xi_0)z}-1) K(z) \phi_\pm(2^{-k_\pm}(y-z))| dz,\quad\text{and}\]
\[\mathbf{II}=\sum_{j=2}^\infty \int_{2^{k_\pm+j}\le \norm{z}\le 2^{k_\pm+j+1}}
|K(z) \phi_\pm(2^{-k_\pm}(y-z))| dz. \]
We first estimate $\mathbf{I}$. Changing variables $z\mapsto \dil{z}{2^{k_\pm+2}}$ we see that
\[ \mathbf{I} \lesssim \int_{\rho(z)\le 1}  |(e^{i\dil{N(x)-\xi_0}{2^{k_\pm+2}}z}-1) K(z)| dz. \]
Using \eqref{eqn:aniso_kernelest}, the previous display is
\[ \lesssim \|m\|_{\mathscr{M}^{\nu_0}} |\dil{N(x)-\xi_0}{2^{k_\pm+2}}|\int_{\rho(z)\le 1}  |z| \norm{z}^{-|\alpha|} dz. \]
Using $\norm{\dil{N(x)-\xi_0}{2^{k_\pm+2}}}=2^{k_\pm+2}\norm{N(x)-\xi_0}\lesssim 1$ we can bound this further as
\[\lesssim  \|m\|_{\mathscr{M}^{\nu_0}}\int_{\rho(z)\le 1} \norm{z}^{1-|\alpha|} dz \lesssim  \|m\|_{\mathscr{M}^{\nu_0}}. \]
This proves that $\mathbf{I}\lesssim  \|m\|_{\mathscr{M}^{\nu_0}}$. It remains to treat $\mathbf{II}$. Here we make use of the fact that we have $\norm{y-z}\ge 2^{k_\pm+j-1}$ in the integrand of $\mathbf{II}$, because of our assumption $\rho(y)\le 2^{k_\pm}$. Using the decay of $\phi_\pm$ we obtain
\[\mathbf{II} \lesssim  \|m\|_{\mathscr{M}^{\nu_0}}\sum_{j=2}^\infty 2^{-j} \int_{2^{k_\pm+j}\le \norm{z}\le 2^{k_\pm+j+1}} \norm{z}^{-|\alpha|} dz\lesssim  \|m\|_{\mathscr{M}^{\nu_0}}. \]
Thus we have proved \eqref{eqn:ac_Rpmest}. The only further ingredient which we need in order to verify \eqref{eqn:ac_errest} is a good estimate for $R_+(y)$ when $2^{k_++1}\le \norm{y}\le 2^{k_-}$. In order to do this we need to do a slightly more careful decomposition. Let us write
\[ Q_\ell = [\dil{\ell}{2^{k_+}},\dil{\ell+1}{2^{k_+}})=\prod_{i=1}^n [2^{k_+\alpha_i} \ell_i, 2^{k_+\alpha_i} (\ell_i+1)) \]
for $\ell\in\Z^n$. Assume that $y\in Q_\ell$ with $1\le|\ell|_\infty< 2^{k_--k_+}$. We have
\begin{equation}\label{eqn:ac_Rplusest1}
R_+(y) \le 2^{-k_+|\alpha|} \sum_{s\in\Z^n} \int_{Q_s} |(e^{i(N(x)-\xi_0)z}-1) K(z) \phi_+(2^{-k_+}(y-z))| dz.
\end{equation}
Moreover, the same estimates that were used to prove \eqref{eqn:ac_Rpmest} yield
\begin{align*}
\int_{Q_s} & |(e^{i(N(x)-\xi_0)z}-1) K(z) \phi_+(2^{-k_+}(y-z))| dz\\
& \lesssim \|m\|_{\mathscr{M}^{\nu_0}} 2^{k_+-k_-} (1+\norm{s-\ell})^{-|\alpha|-1} (1+\norm{s})^{1-|\alpha|}
\end{align*}
Plugging this inequality into \eqref{eqn:ac_Rplusest1} we obtain
\begin{align*}
R_+(y) &\lesssim \|m\|_{\mathscr{M}^{\nu_0}} 2^{-k_-} 2^{k_+(1-|\alpha|)} \sum_{s\in\Z^n} (1+\norm{s-\ell})^{-|\alpha|-1} (1+\norm{s})^{1-|\alpha|}\\
&\lesssim \|m\|_{\mathscr{M}^{\nu_0}} 2^{-k_-} (2^{k_+} \norm{\ell})^{1-|\alpha|},
\end{align*}
where the last inequality requires $\nu$ to be large enough. Therefore,
\begin{equation}\label{eqn:ac_Rplusest2}
R_+(y) \lesssim \|m\|_{\mathscr{M}^{\nu_0}} 2^{-k_-} \norm{y}^{1-|\alpha|}.
\end{equation}
Finally, summarizing \eqref{eqn:ac_errestpfpt1_8}, \eqref{eqn:ac_Rpmest} and \eqref{eqn:ac_Rplusest2} we have shown that
\[ |R(y)| \lesssim \|m\|_{\mathscr{M}^{\nu_0}} (w_0(y) + w_+(y) + w_-(y) + w_1(y)), \]
where
\[ w_0(y) = 2^{-k_-|\alpha|} (2^{-k_-}\norm{y})^{-|\alpha|-1} \mathbf{1}_{\norm{y}\ge 2^{k_-}}, \]
\[ w_\pm (y) = 2^{-k_\pm|\alpha|}\mathbf{1}_{\norm{y}\le 2^{k_\pm+1}},\]
\[ w_1(y) = 2^{-k_-} \norm{y}^{1-|\alpha|}\mathbf{1}_{2^{k_++1}\le \norm{y}\le 2^{k_-}}.\]
Each of these functions is integrable with an $L^1(\R^n)$ norm not depending on $k_-, k_+$, radial and decreasing with respect to $\rho$ in the sense of Lemma \ref{lemma:ac_max} and constant on $\{\rho(y)\le 2^{k_+}\}$ or $\{\rho(y)\le 2^{k_-}\}$. Thus, applying Lemma \ref{lemma:ac_max} to each of these functions yields \eqref{eqn:ac_errest}. Note that to prove \eqref{eqn:ac_errest} we only required that $\nu_0\ge|\alpha|+1$.

\section{Proofs of auxiliary estimates}\label{sec:ac_wavepacketest}

In this section we prove  \eqref{eqn:psiestimate}, \eqref{eqn:tileinteraction} and \eqref{eqn:tileinteraction2}.

\begin{proof}[Proof of \eqref{eqn:psiestimate}]
Expanding definitions and using Fourier inversion we see that, up to a universal constant, $|\psi^N_P(x)|$ is equal to
\[  2^{k_P|\alpha|/2} \left| \int_{\R^n} e^{i\xi(x-c(I_P))} m(\xi-N) \widehat{\phi}(\dil{\xi-c(\omega_{P(0)})}{2^{k_P}}) d\xi \right|. \]
Via a change of variables $\dil{\xi-c(\omega_{P(0)})}{2^{k_P}}\to \zeta$ and using that $m(\xi)=m(\dil{\xi}{2^{k_P}})$ this becomes
\begin{equation}\label{eqn:ac_wavepacketcalc1}
2^{-k_P|\alpha|/2} \left| \int_{\R^n} e^{i\zeta \dil{x-c(I_P)}{2^{-k_P}}} m(\zeta+\dil{c(\omega_{P(0)})-N}{2^{k_P}}) \widehat{\phi}(\zeta) d\zeta \right|.
\end{equation}

Let us fix $x$ and $P$ and take $i$ to be such that $\norm{x-c(I_P)}=|x_i-c(I_P)_i|^{1/\alpha_i}$. From repeated integration by parts we see that \eqref{eqn:ac_wavepacketcalc1} is bounded by
\[ 2^{-k_P|\alpha|/2} (2^{-k_P}\norm{x-c(I_P)})^{-\nu' \alpha_i} \int_{\R^n} \left| \partial^{\nu'}_{\zeta_i} ( m(\zeta+\dil{c(\omega_{P(0)})-N}{2^{k_P}}) \widehat{\phi}(\zeta) )\right| d\zeta, \]
for every integer $\nu'\ge 0$. We set $\nu'=\lceil\nu/\alpha_i\rceil\le \nu$. Since $N\not\in\omega_{P(0)}$ we have $|\zeta+\dil{c(\omega_{P(0)})-N}{2^{k_P}}|\gtrsim 1$. Therefore, 
\[ \int_{\R^n} \left| \partial^{\nu'}_{\zeta_i} ( m(\zeta+\dil{c(\omega_{P(0)})-N}{2^{k_P}}) \widehat{\phi}(\zeta) )\right| d\zeta \lesssim \|m\|_{\mathscr{M}^{\nu'}}\le \|m\|_{\mathscr{M}^{\nu}}. \]
This concludes the proof of \eqref{eqn:psiestimate} in the case that $\norm{x-c(I_P)}\ge 1$. In the case $\norm{x-c(I_P)}\le 1$ we simply use the triangle inequality on  \eqref{eqn:ac_wavepacketcalc1}.
\end{proof}

\begin{proof}[Proof of \eqref{eqn:tileinteraction} and \eqref{eqn:tileinteraction2}]
If $c(I_P)=c(I_{P'})$, the estimates are trivial. Thus we may assume $c(I_P)\not=c(I_{P'})$. We have
\begin{equation}\label{eqn:ac_tileinteract1}
|\langle \phi_P, \phi_{P'}\rangle| \le |I_P|^{-\frac12} |I_{P'}|^{-\frac12} \int_{\R^n} |\phi(\dil{x-c(I_P)}{2^{k_P}})\phi(\dil{x-c(I_{P'})}{2^{k_{P'}}}) | dx.
\end{equation}
Since
\[ \norm{c(I_P)-c(I_{P'})} \le \norm{x-c(I_P)} + \norm{x-c(I_{P'})}, \]
at least one of $\norm{x-c(I_P)}, \norm{x-c(I_{P'})}$ is $\ge \frac12 \norm{c(I_P)-c(I_{P'})}$.
Thus, splitting the integral over $x$ accordingly and using rapid decay of $\phi$, the right hand side of \eqref{eqn:ac_tileinteract1} is no greater than a constant times
\[ |I_P|^{-\frac12} |I_{P'}|^{\frac12} (1+2^{-k_P}\norm{c(I_P)-c(I_{P'})})^{-\nu} + |I_P|^{\frac12} |I_{P'}|^{-\frac12} (1+2^{-k_{P'}}\norm{c(I_P)-c(I_{P'})})^{-\nu}. \]
Recalling that we assumed $|I_P|\ge |I_{P'}|$ we see that the previous display is bounded by a constant times
\[|I_P|^{-\frac12} |I_{P'}|^{\frac12} (1+2^{-k_P}\norm{c(I_P)-c(I_{P'})})^{-\nu}.\]
This proves \eqref{eqn:tileinteraction}.
The estimate \eqref{eqn:tileinteraction2} can be proven in the same way, by using the decay estimate \eqref{eqn:psiestimate}.
\end{proof}

\end{document}